\documentclass[11pt,a4paper,oneside,headings=normal, parskip=half]{article}
\usepackage[left=3cm,right=3cm,top=2cm,bottom=2.5cm]{geometry}			
\usepackage[table,xcdraw]{xcolor}
\usepackage{setspace}	
\usepackage[doi=false,isbn=false,url=false,eprint=false]{biblatex}
\usepackage{aligned-overset}
\usepackage[bottom,hang]{footmisc}										
\usepackage[T1]{fontenc}												
\usepackage[utf8]{inputenc}											
\usepackage[ngerman, english]{babel}									
\usepackage{mathrsfs}
\usepackage{amsmath}													
\allowdisplaybreaks
\usepackage{amsthm}													
\usepackage{amstext}													
\usepackage{amssymb}													
\usepackage{amsfonts}													
\usepackage{dsfont} 													
\usepackage{mathtools}
\usepackage{url}		
\usepackage{xcolor}
\usepackage{graphicx}													
\usepackage{tocloft}													
\usepackage[bf,hang,nooneline,justification=centering]{caption}			
  
\usepackage{tabularx}													
\usepackage{multirow}													
\usepackage{booktabs}													
\usepackage{acronym}
\usepackage[titletoc,title]{appendix}									
\usepackage{biblatex}
\usepackage{placeins}
\usepackage{lmodern} 													
\usepackage{xcolor}														
\usepackage[hidelinks]{hyperref}													
\usepackage{stmaryrd}
\usepackage{comment}

\def\XXint#1#2#3{{\setbox0=\hbox{$#1{#2#3}{\int}$ }
\vcenter{\hbox{$#2#3$ }}\kern-.6\wd0}}

\usepackage{paralist}
\usepackage{enumitem}

 \newcounter{assumption}
\newenvironment{assumption}[1][]{\refstepcounter{assumption}\par\medskip
   \noindent\textbf{Assumption~(A\theassumption). #1} \rmfamily}{\medskip}

\theoremstyle{definition} 
\newtheorem{theorem}{Theorem}[section]                      
\newtheorem{corollary}[theorem]{Corollary} 

\newtheorem{lemma}[theorem]{Lemma}
\newtheorem{definition}[theorem]{Definition}

\newtheorem{Remark}[theorem]{Remark}

\newenvironment{beweis}{\begin{proof}[Proof]}{\end{proof}}

\newtheorem{bemerkung}[theorem]{Remark}

\makeatletter
\newcommand{\AlignFootnote}[1]{%
  \ifmeasuring@
  \else
    \iffirstchoice@
      \footnote{#1}%
    \fi
  \fi}
\makeatother

\DeclareMathOperator{\ent}{Ent}

\DeclareMathOperator*{\fast}{-a.s.}

\newcommand{\dd}{\mathrm{d}}

\newcommand{\E}{\mathbb{E}}

\newcommand{\fastsicher}{\quad \W \fast}

\newcommand{\W}{\mathbb{P}}

\newcommand{\ind}[1]{\mathds{ 1 }_{{#1}}}
\newcommand{\Q}{ \mathbb{Q} }
\newcommand{\F}[2]{\mathcal{F}^{#1}_{#2}}

\newcommand{\skalarq}[1]{ \langle #1 \rangle }

\DeclareMathOperator{\ccc}{C}

\newcommand{\N}{\mathbb{N}}

\newcommand{\R}{\mathbb{R}}

\renewcommand{\phi}{\varphi}
\renewcommand{\epsilon}{\varepsilon}
\newcommand{\eps}{\varepsilon}

\DeclareMathOperator{\llll}{L}
\DeclareMathOperator{\dive}{div}

\newcommand{\abs}[1]{\left\vert #1 \right\vert}
\newcommand{\norm}[1]{\vert\vert {#1}\vert\vert}

\newcommand{\lp}[1]{\llll^{#1}}

\newcommand{\T}{\mathbb{T}}

\newcommand{\PP}{\mathcal{P}}

\graphicspath{ {Bilder/} }
\usepackage{tikz}
\usetikzlibrary{arrows,automata,shapes,calc}
\usepackage{setspace}
\title{Strong Feller Regularisation of $1$-d Nonlinear Transport by Reflected Ornstein--Uhlenbeck Noise\footnote{Supported in part by NNSFC (12531007)}}

\author{Max-K.\ von Renesse\footnote{Universit\"at Leipzig, renesse@math.uni-leipzig.de} \qquad Feng-Yu Wang\footnote{Tianjin University, wangfy@tju.edu.cn} \qquad  
      Alexander Weiß\footnote{Universit\"at Leipzig,  alexander.weiss@math.uni-leipzig.de} \\
}

\begin{document}

\maketitle

\abstract{We consider equations of smooth nonlinear transport on the real line  with regular self-interactions appearing in aggregation models and deterministic mean-field dynamics. We introduce a random perturbation of such systems through a stochastic orientation-preserving flow, which is given by an integrated infinite-dimensional  Ornstein-Uhlenbeck process with reflection. As our main result, we show that the induced stochastic dynamics yield a measure-valued diffusion process on a class of regular measures. Moreover, we establish a quantitative stability of the associated semigroup in relative entropy. As a consequence, the strong Feller property is obtained, as an instance of regularization by noise in case of probability measure valued dynamics. }

\section{Introduction and statement of main results}

This work is inspired by the recent contributions 
\cite{delarue2024rearranged,delarue2024intrinsicregularizationnoise1d} 
to the regularisation by noise phenomenon, 
which is studied there in the case of certain conservative dynamics on the space of measures. Classically, regularisation by noise arises in finite-dimensional ordinary differential equations (ODEs) in various forms. For instance, ODEs with irregular coefficients may admit unique solutions when perturbed or driven by stochastic signals. Other manifestations include improved mixing, the emergence of ergodicity, or enhanced stability of solutions with respect to initial conditions (cf.\ e.g.\ \cite{gess2016regularization, flandoli2011regularization} for an overview). A common explanation for these effects is the additional regularity introduced through diffusion, which in finite dimension is often exploited via PDE methods for a thorough analysis.

In the case of conservative measure-valued dynamical systems, profound new challenges appear when one aims to reproduce similar regularisation effects. First, the powerful tools from PDE theory and its regularity results can typically no longer be used in infinite dimensions. This problem, however, has been successfully addressed over the past years in a number of important cases, which we briefly review in Section \ref{sec:litrev}. Second, the space of probability measures is non-linear (i.e., at best a convex polytope), and so one must find meaningful stochastic perturbations that are on the one hand "strong (i.e., elliptic) enough" and at the same time tangential to the given non-linear state space to yield consistent dynamics.

Conservative deterministic measure-valued dynamics arise as natural macroscopic descriptions in a wide variety of models of very different microscopic origin. Important examples include McKean--Vlasov equations, linear or nonlinear Fokker--Planck dynamics, and mean-field games. In this work, we are guided by the perspective of interpreting them as different models of nonlinear transport with (possibly singular) self-interaction. More specifically, we start from the following model of nonlinear deterministic transport on $\R$
\begin{align}\label{Transportglg}
\begin{split}
\dot \mu_t &= -\operatorname{div} (\mu_t \cdot b_{\mu_t}), \\
\mu_0 &= \mu.    
\end{split}
\end{align}
A standard example is $b(u,\mu) = (\nabla_u \log \mu)(u)$, which induces the heat flow for $\mu$, but below we shall work under rather restrictive assumptions on $b$ allowing only for very regular self-interactions.

In Lagrangian form, the equation becomes  
\begin{align}
\begin{cases}\label{Lagrange}
dx_\mu(u,t) &= b(x_\mu(u,t),\mu_t)\, dt, \\
x_\mu(u,0) &= u, \qquad \forall u \in \R, \\
\mu_t &= \mu \circ x_\mu^{-1}(\cdot,t),
\end{cases}
\end{align}
where $\mu \circ x^{-1}(\cdot,t)$ denotes the pushforward of $\mu$ under the map $x(\cdot,t)$ on $\R$. In fact, assuming smoothness of $b$, the flow $(\mu_t)_{t \ge 0}$ is determined as the unique solution of the nonlinear continuity equation \eqref{Transportglg} with initial condition $\mu$.

Next, since for smooth or coercive $b$ the finiteness of the second moments along the flow is preserved, we identify the measure-valued process $(\mu_t)_{t \ge 0}$ with the $\lp{2}([0,1])$-valued inverse cdfs $(F^\mu_t)_{t \ge 0}$. This yields the equivalent representation  
\begin{align}\label{Wechselwirkung0}
\begin{cases}
dF^\mu_t &= b(F^\mu_t,\mu_t)\, dt, \\
F_0^\mu &= F^\mu, \\
\mu_t &= \lambda \circ (F^\mu_t(\cdot))^{-1}.
\end{cases}
\end{align}
The formal relation between the two representations given by $x_\mu(\cdot,t) = F^\mu_t \circ (F^\mu)^{-1}(\cdot)$. 

Our focus is  on regularisation by noise in the sense of enhanced stability of solutions with respect to the initial conditions, obtained as a result of perturbation by a properly chosen structure-preserving stochastic forcing. In order to regularise the system \eqref{Wechselwirkung0} by adding noise, one needs to ensure that the process remains an inverse cdf. In particular, monotonicity must be preserved. We achieve this by differentiating \eqref{Wechselwirkung0} and adding noise to obtain an SPDE with reflection.

To this aim, we represent a smooth enough   measure $\mu\in \mathcal P_2(\R)$ uniquely in terms of the pair 
\[ \mu \in \mathcal P_2(\R) \simeq (g^\mu, M^\mu),\]
where \[ g^\mu = (F^\mu)', \quad M^\mu = \int_\R x\,\mu(dx)= \int_0^1 F^\mu (s) ds.\]
This leads us to considering the the set of "smooth measures"  
\[
\PP_2^1(\R) = \Big\{ \mu \in \PP_2(\R)\,\Big|\,F^\mu \in H^1((0,1))  \Big\},
\]
which we equip with the natural metric $\rho$ defined by\\
 \[
\rho^2(\mu, \nu) = |M^\mu - M^\nu |^2 + \|F'_\mu - F'_\nu\|^2_{L^2([0,1])}.
\]
Note that  via Poincaré-inequality on $[0,1]$ it holds that 
\[ \gamma_2(\mu, \nu) \leq C \rho(\mu,\nu) \quad \forall \mu, \nu \in \PP(\R)\] for some universal constant $C>0$, where $\gamma_2$ denotes the standard 2-Wasserstein metric on the set $\PP_2(\R)$ of probability measures with finite second moment. In particular,  $(\PP_2^1(\R), \rho)$ is (locally) compactly and densely embedded in $(\PP_2(\R),\gamma_2)$, and each element $\mu \in \PP_2^1(\R)$ is uniquely represented by a pair $(g^\mu,M^\mu)$, where  $g^\mu \in  L^2((0,1))$ is a non-negative square integrable function and $M^\mu \in \R$ is a real number. 

Using this representation,  assuming sufficient smoothness on $b$, we may take the spatial derivative of \eqref{Wechselwirkung0}  to obtain an equivalent system for the description of the $\mu$-dynamics in terms of $(g,M)$, which reads (see section 3 for details)
\[
\begin{cases}
dg_t &= b'(A[(g_t,M_t)],\mu_t)\, g_t \, dt, \\
dM_t &= \int_0^1  b(A[(g_t,M_t)],\mu_t)(x)\, g_t(x)\, dx \, dt, \\
A[(g_t,M_t)](\cdot ) &= \int_0^1 \int^\cdot_u g(r)\, dr\, du + M_t, \\
\mu_t &= \lambda \circ A[(g_t,M_t)]^{-1},\\
g_0 &= \frac{\partial}{\partial u}F^{\mu}, \\
M_0&= M^\mu.
\end{cases}
\]

Finally, in order to produce a meaningful stochastic perturbation that preserves positivity on the level of derivatives, we consider the SPDE system with reflection for the triple $(g,M, \eta)$:
\begin{align}\label{Regularisiert}
\begin{cases}
dg_t &= b'(A[(g_t,M_t)],\mu_t)\, g_t\, dt + \Delta g_t\, dt + dW_t + \eta, \\
dM_t &= \int_0^1 b(A[(g_t,M_t)],\mu_t)(x)\, g_t(x)\, dx\, dt + dB_t, \\
A[(g_t,M_t)](\cdot ) &= \int_0^1 \int^\cdot_u g(r)\, dr\, du + M_t, \\
\mu_t &= \lambda \circ A[(g_t,M_t)]^{-1}, \\
g_t &\ge 0, \\
\langle \eta, g\rangle& =0, 
\end{cases}
\end{align}
which is understood as an SPDE in the Itô sense. Here $\Delta$ is the Laplace operator on $[0,1]$ with Dirichlet boundary conditions, $(W_t)_{t \geq 0}$ is a cylindrical  $L^2([0,1])$-Wiener process, $(B_t)_{t \geq 0}$ is an independent real Brownian motion, and $\eta$ is an adapted random measure on $\R_{\ge 0}\times [0,1]$ enforcing reflection of $g$ at zero to preserve non-negativity of solutions.

As our main result, we show well-posedness of such systems of equations for regular initial data $g_0$ and apply the coupling method to demonstrate the strong Feller regularisation result under strong regularity assumptions on $b$ as follows.

\begin{assumption}\raisebox{\ht\strutbox}{\hypertarget{(A1)}{}}   
 Both $b: \R\times \PP_2(\R) \to \R$ and $b': \R\times \PP_2(\R) \to \R$, where $b'(u, \mu ) = \partial_u b(u,\mu)$ is the partial derivative of $b$ with respect to $u \in \R$, are uniformly bounded and jointly globally Lipschitz-continuous with respect to $|\cdot|$ and $\gamma_2$.
\end{assumption}

Our main results can then be summarised as follows. 

\begin{theorem} \label{thm:mainthm} 
Under Assumption \hyperlink{(A1)}{(A1)}, the system \eqref{Regularisiert} is well-posed for initial conditions $M_0 \in \R$, $g_0\ge0 \in \ccc_0([0,1])$. The family of solutions extends uniquely to a Markov diffusion process on $\mathcal \PP_2^1(\R)$ which is strong Feller. More specifically, for the induced semigroup of kernels on $\PP$ it holds, for all  $\theta\in(0,1)$ and $T>0$
\begin{align*}    \ent&\big( P_T(\cdot,\nu)\vert P_T(\cdot,\mu)\big) \le H_{T,\theta}(\rho_{\mu,\nu},\rho_{\mu})
\end{align*}    
where $\rho_{\mu,\nu}=\rho(\mu,\nu)$, $\rho_\mu = \rho(\delta_{M^\mu},\mu)$, and for some constant $C>0$
\begin{align*}     
     H_{T,\theta}  (\rho_{\mu,\nu},\rho_{\mu})\\
=  C &\Bigg(\frac{1}{T\land 1}\left(\rho^2_{\mu,\nu}+\rho^{\theta}_{\mu,\nu}+\rho_{\mu,\nu}\right)+\left(\rho^{1+\theta}_{\mu,\nu}+\rho^2_{\mu,\nu}\right)\\ &
    +(1+\rho^2_\mu)^{1+\theta}\Big(\frac{T\land 1}{1-\theta}\Big)^{\theta}\big(\log(1+\rho^{-1}_{\mu,\nu})\big)^{-\theta}\Bigg),
\end{align*}
\end{theorem}

\begin{Remark}
1) The term $\frac 1 T \rho^2_{\mu, \nu}$ on the right hand side of the entropy estimate, for small $T>0$ features a Gaussian-like behaviour, which in the current case, of course, cannot be expected in clean form. \\
2) Via Pinsker's inequality 

one concludes from the estimate above that  for any  measurable $F:\PP_2^1 \to \R$, that 
\[ |P_TF(\mu)- P_TF(\nu)| \leq \|F\|_\infty  \sqrt {\frac {1}{2} H_{T,\theta}(\rho_{\mu,\nu},\rho_{\mu}\land\rho_\nu)}.\]
 This is the the strong Feller property in quantitative form. In particular, the map  $\mu \to P_TF(\mu)$ is continuous on $\PP^1_2(\R)$, locally uniformly with respect to the metric $\rho$.

3) A classical example of a $b$ satisfying condition \hyperlink{(A1)}{(A1)} is the McKean--Vlasov interaction
\[
b(u,\mu) = \int_\R h(u-v)\, \mu(dv),
\]
where $h \in C_c^\infty(\R)$ is a smooth kernel function.  

4) Theorem \ref{thm:mainthm} implies in particular that the induced Markov process $(\mu_t^{\mu})_{t \ge 0}^{\mu \in \PP_2^1(\R)}$ has the strong Feller property on $\PP_2^1(\R)$. As mentioned above, this property represents an instance of regularisation by noise. In fact, for purely deterministic systems such a statement is false in general, even if the data of the ODE are smooth. The underlying mechanism for such a regularisation effect is the possibility to translate the perturbation of the initial condition inside the expectation into a well-behaved perturbation of the stochastic signal, provided that the set of admissible shifts for quasi-invariance of the underlying probability space is rich enough. This principle lies at the heart of the phenomenon and typically involves a Girsanov transform argument. In Section \ref{sec:strongfellerprop} we will pursue this strategy in the present setting accordingly.
\end{Remark}

The remainder of the paper is organised as follows. A review of the relevant literature and predecessors is given in Section \ref{sec:litrev}. The setting and notation are introduced in Section \ref{sec:setting}. Section \ref{sec:wellposed} is devoted to the well-posedness of the system \eqref{Regularisiert}. Finally, the proof of our main theorem is presented in Section \ref{sec:strongfellerprop}.

\section{Literature and Previous Results} \label{sec:litrev}

The main inspiration for this work comes from the recent breakthrough contribution of Delarue and Hammersley \cite{delarue2024rearranged}, where a strong Feller diffusion on $\PP(\R)$ was constructed by means of a novel symmetrization (¨rearrangement¨) mechanism applied to an infinite dimensional OU-process over $\R$. Using this model for stochastic perturbation, a strong well-posedness result in 1-d was obtained in \cite{delarue2024intrinsicregularizationnoise1d} in the context of  mean-field games. Previous works in this direction  was carried out by Marx \cite{marx1, marx2} using a particle-based stochastic regularization, motivated by the construction of the Wasserstein diffusion \cite{von2009entropic,andres2010particle}. Our construction below is very similar in spirit to the above mentioned work \cite{delarue2024rearranged}, but uses a more explicit reflection mechanism, which allows for a simplified and arguably slightly more conventional coupling procedure.

The common idea for all these works is to establish or use regularization by noise, here in the setting of measure valued dynamics. For a comprehensive overview of the regularisation-by-noise phenomenon, mostly in the finite-dimensional setting, we refer to the review articles \cite{gess2016regularization, flandoli2011regularization}. Crucial progress in finite-dimensional non-Markovian or pathwise settings beyond the classical Krylov–Röckner framework \cite{MR2117951} was made by Friz and Cass \cite{MR2680405} and many subsequent works in the spirit of the rough-path framework, with exciting new developments based on Gubinelli’s fundamental sewing lemma \cite{MR2091358} and Le’s extension \cite{MR4089788} to the stochastic case; cf.\ e.g.\ \cite{MR4930606,MR4720220}. In infinite-dimensional settings, regularisation by noise reveals how stochastic perturbations can restore uniqueness or improve well-posedness of ill-posed PDEs. For instance, Flandoli, Gubinelli, and Priola demonstrated pathwise uniqueness for a stochastically perturbed transport equation despite deterministic non-uniqueness \cite{MR2593276}. Hairer and Mattingly further established strong Feller and ergodicity for the stochastic Navier–Stokes system in a hypoelliptic framework \cite{MR2259251}; see also \cite{MR3825883} for the strong Feller property for singular SPDEs. 

\nocite{peszat1995strong,flandoli2011regularization,zhang2010white}

Systems of the type \eqref{Lagrange} with additional stochastic forcing $\sigma(x_\mu(u,t),\mu_t)\, dW_t$ for finite-dimensional Brownian motion $W$ were introduced by A.\ A.\ Dorogovtsev on $\R^d$ in \cite{AAD1} under the name \textit{SDEs with interaction} and have been studied intensively ever since (see, e.g., \cite{dorogovtsev2023measure, DorWeiss, gess2022conservativespdesfluctuatingmean}). In spite of the structural similarity to McKean–Vlasov equations, we emphasize that the measures $\mu_t$ above are image measures under a self-induced flow, whereas in the McKean–Vlasov case they represent the time-evolving laws, i.e.\ statistical averages. In particular, in the extended Dorogovtsev system \eqref{Lagrange} with noise, the measure-valued process $(\mu_t)$ is random, contrary to the McKean–Vlasov case. Hybrid models were recently investigated, e.g., in \cite{wang2021image}.

\smallskip

Measure-valued processes of this type have been heavily studied in the last five years. Wang, who, independently of Dorogovtsev \cite{AAD1,Dorogovtsev+2024}, reintroduced the notion of SDEs with interaction on Euclidean spaces under the name \textit{image-dependent SDEs} in \cite{wang2021image} and studied their properties, most notably the semigroup properties of the measure-valued process and smoothness of solutions with respect to the initial measure. Furthermore, the connection between SDEs with interaction and McKean–Vlasov equations with common noise, and therefore mean-field games, was established. Wang and Ren studied regularity properties of McKean–Vlasov equations (with common noise) in \cite{ren2019bismut}, and Huang and Ren in \cite{huang2021distribution}, and Huang in \cite{huang2023regularitiesexponentialergodicityentropy}. There have also been considerations in the context of machine learning by Gess and Konarovskyi \cite{gess2022conservativespdesfluctuatingmean} and by Gess, Kassing, and Konarovskyi \cite{gess2023stochasticmodifiedflowsmeanfield}, where the authors show that the measure-valued process induced from SDEs with interaction solves an SPDE arising in machine learning.

The investigation of reflected SPDEs was initiated by Nualart, Pardoux, and Donati-Martin in their groundbreaking works \cite{nualart1992white, donati1993white} by proving existence and uniqueness properties, first through deterministic obstacle problems \cite{nualart1992white} and then through a SPDE penalisation approach \cite{donati1993white}. The existence and uniqueness results were later generalised by Xu and Zhang \cite{xu2009white}, who also proved a large deviations principle for the equation. Properties of the semigroup were investigated by Zambotti \cite{zambotti2001reflected}, who showed that the SPDE with reflection admits an invariant measure represented by the Bessel bridge. Zhang proved the strong Feller property and a Harnack inequality in the case of non-functional coefficients \cite{zhang2010white}. The main technical tool for proving Theorem~\ref{thm:mainthm}, namely the method for showing regularity properties of semigroups induced by S(P)DEs, has been extensively developed by Wang (see, e.g., \cite{wang2014analysis} and many other works). The first application of this method to reflected SPDEs was carried out by Zhang in \cite{zhang2010white} to establish the Harnack inequality. The strong Feller property for reflected SPDEs was also shown there, but only in the case with non-functional coefficients.
 
\section{Setting and basic notation} \label{sec:setting}

Let $\PP_2(\R)$ be the space of all measures on $\R$ with finite second moment, equipped with the quadratic Wasserstein distance
\begin{align*}
    \gamma^2_2(\mu,\nu)= \inf_{\kappa \in C(\mu,\nu)} \int_{\R\times \R}  \abs{u-v}^2 \kappa(\dd u,\dd v),
 \end{align*}
and let $(\PP_2(\R),\gamma_2)$ denote the corresponding $2$-Wasserstein space.
In one spatial dimension it is a classical fact that 
the map 
\begin{align*}
    \chi: G  &\to \PP_2(\R), \\
    F^\mu &\mapsto \lambda \circ (F^\mu(\cdot))^{-1}=\mu,  \end{align*}
is a bijective isometry,
where $F^\mu$ is the generalized right-inverse CDF (quantile function)  of $\mu$ and  
\[
G=\{F\in \lp{2}([0,1]): F \text{ is non-decreasing}\}.
\]

Furthermore, let $\ccc_0([0,1])$ and $\lp{2}([0,1])$ denote the space of continuous functions with zero boundary values or square-integrable functions on $[0,1]$, equipped with the norms 
\begin{align*}
    \norm{f}_\infty&= \sup_{x\in[0,1]} \abs{f(x)},\\
    \norm{f}^2_{\lp{2}}&= \int_0^1 \abs{f(x)}^2 \dd x.
\end{align*}
 
We aim to regularise the measure-valued process induced by the following transport equation: 
\begin{align}\label{transport}
    \begin{cases}
        d\mu_t&= -\dive(b(\cdot,\mu_t)\mu_t)\dd t, \\
        \mu_0 &= \mu \in \PP_2(\R),
    \end{cases}
\end{align}
where $b: \R \times \PP_2(\R)\to \R$. Under mild regularity assumptions this equation turns out to take values in the nonlinear Wasserstein space, which via $\mu \mapsto F^\mu $ is mapped  isometrically to a convex subset of $\lp{2}([0,1])$.

In the first step, we derive an equation for the inverse CDF $(F^\mu_t)_{t\ge 0}$ corresponding to $(\mu_t)_{t\ge 0}$. Note that the transport equation \eqref{transport} corresponds to the following Lagrangian equation: 
\begin{align*}
   \begin{cases}
     dx_\mu(u,t) &= b(x_\mu(u,t),\mu_t)\dd t, \\ 
     x_\mu(u,0) &= u\in \R, \\
     \mu_t &= \mu \circ x^{-1}_\mu(\cdot,t).
   \end{cases}
\end{align*}
 
Under the assumption that $b$ is Lipschitz, these equations are well posed.

We now derive an equation for the inverse CDF $(F^\mu_t)_{t\ge 0}$ of $(\mu_t)_{t\ge 0}$. Let $\lambda$ be the Lebesgue measure on $[0,1]$. Then 
\begin{align*}
    \lambda \circ (F^\mu_t(\cdot))^{-1}
    &=\mu_t 
    = \mu \circ x^{-1}_\mu(\cdot,t)
    = \lambda \circ (F^\mu(\cdot))^{-1} \circ x^{-1}_\mu(\cdot,t)
    = \lambda\circ (x_\mu(F^\mu(\cdot),t))^{-1}.
\end{align*}
From regularity assumptions on $b$ we can deduce that $x_\mu(F^\mu(\cdot),t)$ is increasing, and thus, by uniqueness of the inverse CDF, we conclude $F^\mu_t(\cdot)= x_\mu(F^\mu(\cdot),t)$. Hence 
\begin{align*}
    \begin{cases}
    d F^\mu_t &= b(F^\mu_t,\mu_t)\dd t, \\     
    F_0^\mu&= F^\mu, \\
    \mu_t&= \lambda \circ (F^\mu_t(\cdot))^{-1}.
    \end{cases}
\end{align*}

Now that we have derived an equation for $(F^\mu_t)_{t\ge 0}$, we want to regularise it using noise. However, the regularised solution must remain increasing.  We therefore consider the equation for the derivative $\frac{\partial}{\partial u} F = g$, since regularisation at that level only requires enforcing positivity. This yields the equation
\begin{align*}
\begin{cases}
    d g_t &= b^\prime(F_t,\lambda \circ (F^\mu_t(\cdot))^{-1})\, g_t \dd t,\\
    g_0&= \frac{\partial}{\partial u}F^\mu = g^\mu.
\end{cases}    
\end{align*}
To go in inverse direction we write 
\begin{align*}
    F_t(u)-F_t(v)= \int_v^u g_t(r) \dd r,
\end{align*}
and thus  
\begin{align*}
    F_t(u) 
    &= \int_0^1 \int_v^u g_t(r)\, \dd r \dd v 
      + \int_0^1 F_t(v)\, \dd v
    =: A[(g_t,M_t)](u),
\end{align*}
where we define $M_t := \int_0^1 F_t(v)\dd v$. For later reference note that 
\begin{align}\label{AufleitungLip}
\begin{split}
    \abs{A[(\phi,x)](z)-A[(\psi,y)](z)}
    &\le \int_0^1 \int_y^z \abs{\phi(r)-\psi(r)} \dd r \dd y+\abs{x-y}\\
    &\le \norm{\phi-\psi}_{\lp{2}} +\abs{x-y}
    \le \norm{\phi-\psi}_\infty +\abs{x-y}.
\end{split}
\end{align}

As a consequence we obtain
\begin{align*}
   \begin{cases}
         d g_t &= b^\prime(A[(g_t,M_t)],\mu_t)\, g_t \dd t,\\
         d M_t&= \int_0^1 b(A[(g_t,M_t)](u),\mu_t)\dd u \, \dd t,\\
         g_0&=\phi\ge 0,\\
         M_0&= x, \\
         \mu_t&= \lambda \circ \left(A[(g_t,M_t)](\cdot)\right)^{-1},
   \end{cases}
\end{align*}
where $\phi\in \lp{2}([0,1])_{\ge 0}$ and $x\in \R$.

We now add noise to the equation in such a way that the derivative remains positive. This is achieved by means of the following reflected SPDE:

\begin{align}\label{System}
\begin{cases}
    \dd g_t&= \Delta g_t +b^\prime(A[(g_t,M_t)],\mu_t)\, g_t\dd t + \dd W_t+ \eta,\\
    \dd M_t &= \int_0^1 b(A([g_t,M_t](x)),\mu_t)\dd x\, \dd t + \dd B_t,\\
    \langle \eta, g\rangle &=0, \\
    (g_0,M_0)&= (\phi,x),\\
    g_t(0)&=g_t(1)=0,\\
    g_t&\ge 0,\\
    \mu_t&= \lambda \circ \left(A[(g_t,M_t)](\cdot)\right)^{-1},
\end{cases}    
\end{align}
where $dW$ is $L^2([0,1])$-space-time white noise and $B$ is a real Brownian motion independent of $W$.

\section{Existence and uniqueness} \label{sec:wellposed}
The existence and uniqueness result will follow from a more general result for locally Lipschitz coefficients with at most linear growth which we show in this section. Before proving this, we present two lemmas due to \cite{nualart1992white} which provide well-posedness and crucial estimates for a deterministic variational inequality that we will need. 
\begin{lemma}\label{hindernisproblem}
 Let $\Delta$ be the Dirichlet Laplacian  on $[0,1]$, and let $v:[0,1]\times\R_+ \to \R$ be continuous with respect to $(t,x)$ and $v(\cdot,t)\in \ccc_0([0,1])$ for all $t\ge 0$. Furthermore assume $v(0,x)\ge 0$ for all $x\in [0,1]$. Then there exists a unique pair $(z,\eta)$ such that:
 \begin{enumerate}[label=\roman*)]
     \item $z(0,t)=z(1,t)=0$, $z\ge -v$, and $z(x,0)=0$ for all $x\in [0,1]$;
     \item $\eta$ is a measure on $(0,1)\times\R_+$ such that 
     \[
     \eta([\eps,1-\eps] \times [0,T])<\infty
     \]
     for all $\eps,T>0$;
     \item for all $t\ge 0$ and $\phi\in \ccc^\infty_c((0,1))$,
     \begin{align*}
         \skalarq{z_t,\phi}_{\lp{2}} -\skalarq{z_0,\phi}_{\lp{2}}
         = \int_0^t \skalarq{z_s,\Delta\phi}_{\lp{2}} \dd s
         + \int_0^t\int_0^1 \phi(x)\, \eta(\dd x,\dd s);\\
     \end{align*}
     \item $\langle \eta, z+v\rangle =0$.
 \end{enumerate}
\end{lemma}
\begin{proof}
  Theorem 1.4 of \cite{nualart1992white}.
\end{proof}

\begin{lemma}\label{Hindernisabsch}
    In the situation of Lemma \ref{hindernisproblem}, let $v,\Tilde{v}\in \ccc(\R_+\times[0,1])$ such that $v(t),\Tilde{v}(t)\in \ccc_0([0,1])$, and consider the corresponding solutions $(z,\eta)$ and $(\Tilde{z},\Tilde{\eta})$. Then:
    \begin{enumerate}[label=\roman*)]
        \item For all $T>0$,
        \[
            \sup_{0\le t\le T} \norm{z_t}_{\infty}\le \sup_{0\le t \le T }\norm{v_t}_{\infty};
        \]
        \item For all $T>0$,
        \[
            \sup_{0\le t \le T} \norm{z_t-\Tilde{z}_t}_{\infty}
            \le \sup_{0\le t\le T} \norm{v_t-\Tilde{v}_t}_{\infty}.
        \]
    \end{enumerate}
\end{lemma}
\begin{proof}
     Theorem 1.3 in \cite{nualart1992white}.
\end{proof}

We now prove existence and uniqueness for a more general class of coefficients. First, we define what we mean by a solution to systems of type \eqref{System}. To this aim, we consider the class of systems for $(g,M, \eta)$
\begin{align}\label{System2}
    \begin{cases}
        \dd g_t &= \Delta g_t + a(g_t,M_t)\dd t + \dd W_t + \dd \eta, \\
        \dd M_t &=  \Tilde{a}(g_t,M_t)\dd t + \dd B_t,\\
        \langle \eta, g\rangle &=0\\
        g_0&\ge 0,\quad g_0(0)=g_0(1)=0,\quad g_t(0)=g_t(1)=0, \\
        M_0&\in \R,
    \end{cases}
\end{align}
where $a: \Omega \times \ccc(\T ) \times \R\times [0,\infty) \to \ccc_0([0,1])$ and $\Tilde{a}:\Omega \times \ccc_0([0,1]) \times \R \times [0,\infty) \to \R$ are locally Lipschitz with at most linear growth for all $\omega\in \Omega$. That is,
\begin{align}\label{loklip}
\begin{split}
    \norm{a(\omega,\phi,x,t)-a(\omega,\psi,y,t)}_\infty
    &\le C_n(T)(\norm{\phi-\psi}_\infty + \abs{x-y}), \\
    \abs{\Tilde{a}(\omega,\phi,x,t)-\Tilde{a}(\omega,\psi,y,t)}
    &\le C_n(T) (\norm{\phi-\psi}_\infty + \abs{x-y}),
\end{split}
\end{align}
for all $\omega \in \Omega$, $T>0$, and $t\le T$, whenever $\norm{\phi-\psi}_\infty + \abs{x-y}\le n$. Furthermore, we assume the coefficients to be of at most linear growth:
\[
    \norm{a(\omega,\phi,x,t)}_\infty + \abs{\Tilde{a}(\omega,\phi,x,t)}
    \le C(T)(1+\norm{\phi}_\infty)
\]
for all $\omega\in \Omega$ and $t\le T$.

\begin{definition}\label{Loesung}
    A triple $(g,M,\eta)$ is called a solution to \eqref{System2} if: 
    \begin{enumerate}[label=\roman*)]
        \item $(g,M)= \{(g_t(u),M_t): (u,t)\in [0,1] \times \R_+ \} $ is a continuous adapted process, where $g$ is nonnegative with $g_t(0)=g_t(1)=0$; 
        \item $\eta(\dd x,\dd t)$ is a random measure on $[0,1]\times \R_+$
        such that $\eta((\eps,1-\eps)\times [0,T])<\infty$ almost surely for all $T,\eps> 0$, and $\eta$ is adapted (i.e.\ $\eta(B)$ is $\F{}{t}$-measurable whenever $B\in \mathcal{B}([0,1]\times[0,t])$); 
        \item For all $t\ge 0$ and $\phi\in \ccc_c((0,1))$,
        \begin{align*}
            \skalarq{g_t,\phi}-\int_0^t \skalarq{g_s,\Delta\phi}\dd s
            &+ \int_0^t \skalarq{a(g_s,M_s),\phi}\dd s
            = \skalarq{g_0,\phi}
            +\int_0^t\int_0^1 \phi(x)\,W(\dd x,\dd s)\\
            &\qquad +\int_0^t\int_0^1 \phi(x)\,\eta(\dd x,\dd s) 
            \quad\text{a.s.},\\
            M_t&= M_0 + \int_0^t  \Tilde{a}(g_s,M_s)\dd s + B_t;
        \end{align*}
        \item $\displaystyle \int_Q g \dd \eta= 0$, where $Q=[0,1]\times \R_+$.
     \end{enumerate}
\end{definition}
\begin{theorem}\label{Wohlgestellt} 
   There exists a unique solution $(g,M,\eta)$ to \eqref{System2} with initial condition $(\phi,x)\in \ccc_0([0,1])_{\ge 0}\times \R$  such that $g_t\in \ccc_0([0,1])_{\ge 0}$ almost surely, if
   \begin{enumerate}[label=\roman*)]
       \item $a$ and $\Tilde{a}$ are locally Lipschitz in the sense of \eqref{loklip};
       \item $a$ and $\Tilde{a}$ are of at most linear growth,
   \end{enumerate} 
  holds.
\end{theorem}

\begin{proof}
   
    The proof is analogous to \cite{xu2009white} and is based on successive approximation. First, we assume global Lipschitzness to hold for $a$ and $\Tilde{a}$. Take $g_0\in \ccc_0([0,1])_{\ge 0}$, define  
    \begin{align*}
        f_t^1 (x)
        &= \int_0^1 G_t(x,y)g_0(y) \dd y 
        - \int_0^t\int_0^1 G_{t-s}(x,y)a(g_0,M_0)(y)\dd y \dd s \\
        &\quad + \int_0^t \int_0^1 G_{t-s}(x,y)W(\dd y,\dd s),
    \end{align*}
    where $G$ denotes the Green function of $\Delta$ with Dirichlet boundary conditions.

    Then $f^1$ solves the following equation (see, e.g., \cite{da2014stochastic}):
    \begin{align*}
        d f^1_t(x)= \Delta f^1_t(x) + a(g_0,M_0)\dd t + \dd W(x,t), \qquad f_0=g_0.
    \end{align*}
    Moreover, let $(z^1,\eta^1)$ be the solution according to Lemma \ref{hindernisproblem} with $v=f^1_t$, note that $f^1$ is continuous. Then 
    $g_t^1 = f_t^1 + z^1$ solves 
    \begin{align*}
        d g_t^1(x)=  \Delta g^1_t(x) +a(g_0,M_0)\dd t + \dd W(x,t)+ \eta^1, 
        \qquad g^1_0=g_0. 
    \end{align*}
    Furthermore, let 
    \begin{align*}
            \dd M_t^1&= \Tilde{a}(g_t^1, M_t^1)\dd t +\dd B_t\\
              M_0^1 &=M_0,
    \end{align*}
    where the solution exists uniquely since we assumed that the coefficients are Lipschitz and bounded. Now let $f^n$ be defined as
    \begin{align*}
        f^n_t (x)
        &= \int_0^1 G_t(x,y)g_0(y
        ) \dd y 
        - \int_0^t\int_0^1 G_{t-s}(x,y)a(g_s^{n-1},M^{n-1}_s)(y) \dd y \dd s \\
        &\quad + \int_0^t \int_0^1 G_{t-s}(x,y)W(\dd y,\dd s),
    \end{align*}
    and let $(z^n,\eta^n)$ be a solution according to Lemma \ref{hindernisproblem} with $v=f^n$ where we note that $f^n$ is continuous.
    Then $g^n_t=f_t^n+z^n_t$ solves 
    \begin{align*}
       d g_t^n(x)=  \Delta g^n_t(x) +a(g_t^{n-1},M_t^{n-1})\dd t + \dd W(x,t)+ \eta^n, 
       \qquad g^n_0=g_0,
    \end{align*}
    
    and 
    \begin{align}
    \begin{split}
       dM^n_t&= \Tilde{a}(g_t^n,M_t^n)\dd t +\dd B_t.\\
       M_0 &= M_0
       \end{split}
    \end{align}

    To show that these sequences converge we use Lemma \ref{Hindernisabsch} to find that 
    \begin{align*}
       \sup_{0\le t\le T}\norm{z^n_t-z_t^{n-1}}_\infty 
       \le \sup_{0\le t\le T} \norm{f^n_t-f^{n-1}_t}_\infty,
    \end{align*} 
    hence 
    \begin{align*}
        \sup_{0\le t\le T}\norm{g^n_t-g_t^{n-1}}_\infty 
        \le C\left(\sup_{0\le t\le T} \norm{f_t^n-f_t^{n-1}}_\infty\right).
    \end{align*}
    Furthermore, 
\begin{align*}
        \E\Big(\sup_{0\le t\le T}\abs{M_t^{n}-M_t^{n-1}}^p\Big)         \le
        C \left(\int_0^T \E\big(\abs{M^{n}_t-M_t^{n-1}}^p\big) \dd t 
        + \int_0^T \E\big(\norm{g^n_t-g^{n-1}_t}_{\infty}^p\big) \dd t  \right).
\end{align*}

    Hence, by Grönwall's inequality we obtain 
    \begin{align*}
         \E\Big(\sup_{0\le t\le T}\abs{M_t^{n}-M_t^{n-1}}^p\Big)
         \le C\int_0^T \E\big(\norm{g^n_t-g^{n-1}_t}_{\infty}^p\big) \dd t.  
    \end{align*}

    Therefore, we can now resume
    \begin{align*}
        &\E\Big(\sup_{0\le t \le T} \norm{g^n_t-g_t^{n-1}}_{\infty}^p\Big) \\
        \le\;& C \E\left(\sup_{x\in [0,1],\,0\le t\le T} 
        \abs{\int_0^t\int_0^1 G_{t-s}(x,y)\big[a(g^{n-1}_s),M^{n-1}_s)
        -a(g^{n-2}_s),M^{n-2}_s)\big]\dd y\dd s}^p\right).
    \end{align*}

    Let $p>1$ be large enough such that its conjugate exponent $q<3$.  Then 
    \begin{align*}
        \E\Big(\sup_{0\le t \le T} \norm{g^n_t-g_t^{n-1}}^p_\infty\Big)
        &\le C \E\left(\sup_{x\in [0,1],\,0\le t\le T} 
        \abs{\int_0^t \int_0^1 G^q_s(x,y) \dd y \dd s} \right)^{\frac{p}{q}} \\
        &\quad \times \E\left(\int_0^T  \norm{g_t^{n-1}-g_t^{n-2}}_\infty^p 
        +\abs{M_t^{n-1}-M_t^{n-2}}^p \dd t\right) \\
        &\le C \E\left(\int_0^T  \norm{g_t^{n-1}-g_t^{n-2}}_\infty^p \dd t\right).
    \end{align*}

    Hence, one can now prove that $(g^n_t,M_t^n)$ converges in $\lp{p}(\Omega, \ccc(\T \times [0,T])\times \R)$. Denote the limits by $(g_t,M_t)$. We now show that we actually have a solution to \eqref{System}. First, note that since $g^n_t(x)\ge 0 $ almost surely, we obtain $g_t(x)\ge 0$ almost surely. For any $\phi\in\ccc_0([0,1])$ and $n\in \N$ we have 
    \begin{align*}
        \skalarq{g^n_t,\phi} 
        &-\int_0^t\skalarq{g^n_s,\Delta \phi} \dd s \\
        &\quad +\int_0^t \skalarq{a(g^{n-1}_s,M^{n-1}_s),\phi} \dd s\\
        &= \skalarq{g_0,\phi}
        + \int_0^t\int_0^1 \phi(x) W(\dd x,\dd s)
        + \int_0^t\int_0^1 \phi(x)\eta^n(\dd x,\dd s).
    \end{align*}

    Since the left-hand side converges as $n\to \infty$, we obtain that $\eta^n$ converges to a positive distribution, thus making it a measure. Therefore, one can show iii) in Definition \eqref{Loesung}. Property iv) can be proven in exactly the same way as in \cite{xu2009white}, Theorem 2.1. Note that the solution satisfies the following bound:
    \begin{align*}
        \E\Big(\sup_{0\le t\le T} \norm{g_t}^p_\infty+ \abs{M_t}^p\Big)\le C,
    \end{align*}
    since we assume the coefficients to be of at most linear growth. Hence, one can proceed by a standard localisation argument to deduce the existence result for locally Lipschitz coefficients.

    We shall now prove uniqueness. Let $(g^1,M^1,\eta^1)$ and $(g^2,M^2,\eta^2)$ be solutions to \eqref{System2}, and define 
    \begin{align}\label{aufspaltung}
        f^i_t(x)
        &= \int_0^1 G_t(x,y)g^1_t(y) \dd y 
        + \int_0^t \int_0^1 G_{t-s}(x,y)a(g^1_s,M^1_s)(y) \dd y \dd s \\
        &\quad + \int_0^t \int_0^1G_{t-s}(x,y) W(\dd y,\dd s). 
    \end{align}
    Then $z^i= g^i-f^i $ is the unique solution to the problem in Lemma \ref{hindernisproblem} with $v^i= g^i$. Hence, setting $\tau_N:= \inf\{ t\ge 0: \norm{g^1_t}_\infty+\norm{g^2_t}_\infty+\abs{M^1_t}+\abs{M^2_t}\le N \}$, we obtain, similarly as above,
    \begin{align*}
        &\E\Big(\sup_{0\le t\le T\land \tau_N}\norm{g_t^1-g^2_t}_\infty^p 
        +\abs{M_t^1-M^2_t}^p\Big) \\
        \le\;& C \E\left(\int^{T\land \tau_N}_0
        \norm{f^1_t-f^2_t}_\infty^p
        +\norm{z^1_t-z^2_t}^p_\infty 
        + \abs{M^1_t-M^2_t}^p\dd t \right)\\
        \le\;& C  \E\left(\int_0^{T\land\tau_N}
        \norm{g^1_t-g^2_t}_\infty^p
        + \abs{M^1_t-M^2_t}^p\dd t\right).
    \end{align*}
    By Grönwall's inequality and letting $N\to \infty$, we get $g^1= g^2$ and  $M^1=M^2$ almost surely.
\end{proof}

\section{Strong Feller property} \label{sec:strongfellerprop}

In this section we discuss the following equation 
\begin{align}\label{System3}
\begin{cases}
    \dd g_t&= \Delta g_t +b^\prime(A[(g_t,M_t)],\mu_t) g_t \dd t + \dd W_t+\dd \eta,\\
    \dd M_t &= \int^1_0 b(A[(g_t,M_t)](z),\mu_t) \dd z \dd t + \dd B_t, \\
    (g_0,M_0)&= (\phi,x),\\
    g_t(0)&=g_t(1)=0,\; g_t\ge 0,\\
    \langle \eta, g\rangle &=0,
\end{cases}    
\end{align}
for $(\phi,x)\in \ccc_0([0,1])\times \R$. From now on, we denote by $(g^{\phi,x}_t,M^{\phi,x}_t)$ a solution to \eqref{System3} with initial condition $(\phi,x)\in \ccc_0([0,1]) \times \R$.

For that matter we require the following conditions on $b$.

\begin{assumption}\raisebox{\ht\strutbox}{\hypertarget{(A2)}{}} 
     For all $(u,\mu),(v,\nu)\in \R\times \PP_2(\R)$ there exists $C>0$ such that 
\begin{align*}
         \abs{b(u,\mu)-b(v,\nu)}+ \abs{b^\prime(u,\mu)-b^\prime(v,\nu)}\le C\bigl(|u-v|+ \gamma_2(\mu,\nu)\bigr)
     \end{align*}
     
\end{assumption}

\begin{assumption}\raisebox{\ht\strutbox}{\hypertarget{(A3)}{}} 
       For all $(u,\mu)\in \R\times \PP_2(\R)$, we have 
     \begin{align*}
         \abs{b(u,\mu)}+\abs{b^\prime(u,\mu)}\le C
     \end{align*}
     for some $ C>0$.
\end{assumption}

Under these assumptions, the coefficients of \eqref{System3} satisfy the assumptions of Theorem \ref{Wohlgestellt}, since by \eqref{AufleitungLip} we get 
 \begin{align*}
    &\abs{b^\prime(A[(\phi,x)](u),\lambda \circ A^{-1}([\phi,x])(\cdot))-b^\prime(A[(\psi,y)](u),\lambda \circ A^{-1}([\psi,y])(\cdot))}\\
    \le\; & C \left(\abs{A[(\phi,x)](u)-A[(\psi,y)](u)} + \left(\int_0^1 \abs{A[(\phi,x)](u)-A[(\psi,y)](u)}^2 \dd u \right)^{\frac{1}{2}}\right)\\
    \le\; & C(\norm{\phi-\psi}_\infty + \abs{x-y}).
 \end{align*} 
 Moreover, 
 \begin{align*}
     &\abs{\int_0^1 b(A[(\phi,x)](z),\lambda\circ (A[(\phi,x)](\cdot))^{-1}) \dd z - \int_0^1 b(A[(\psi,y)](z),\lambda\circ (A[(\psi,y)](\cdot))^{-1}) \dd z }\\
     \le\; & \int_0^1 \abs{ b(A[(\phi,x)](z),\lambda\circ (A[(\phi,x)](\cdot))^{-1})-b(A[(\psi,y)](z),\lambda\circ (A[(\psi,y)](\cdot))^{-1})} \dd z\\
     \le\; & C \int_0^1 \left(\abs{A[(\phi,x)](z)-A[(\psi,y)](z)} + \left(\int_0^1 \abs{A[(\phi,x)](u)-A[(\psi,y)](u)}^2 \dd u \right)^{\frac{1}{2}}\right) \dd z\\
     \le\; & C \left(\norm{\phi-\psi}_\infty + \abs{x-y}\right).
 \end{align*}
 As a consequence, we can deduce that the coefficients are locally Lipschitz in the sense of \eqref{loklip}, and by boundedness of $b$ and $ b^\prime$ they are also of at most linear growth. 
 Thus we have existence and uniqueness for the system \eqref{System3}.

We now want to prove the strong Feller property for the SPDE with reflection. In the case of non-functional coefficients, this has already been done in \cite{zhang2010white} via the coupling method, which we will also use. However, in our case we work with a functional dependence of the coefficients such that the arguments given there do not apply directly. We overcome this by a method of freezing the coefficients. To this aim we fix the solution $(g^{\phi,x}_t,M^{\phi,x}_t)$ and denote 
 \[
 b^\prime(A[(g^{\phi,x}_t,M^{\phi,x}_t)](z),\mu_t^{\phi,x})=\beta(\phi,x,t)(z),
 \]
with  
\[\mu_t^{\phi,x}= \lambda \circ \left(A[(g^{\phi,x}_t,M^{\phi,x}_t)](\cdot)\right)^{-1}.\]

 We then consider the penalised equations 
 \begin{align}\label{penalisierung}
     \begin{cases}
         \dd g^{\eps,\phi,x}_t &= \Delta g^{\eps,\phi,x}_t \dd t + \beta(\phi,x,t)g^{\eps,\phi,x}_t \dd t +\dd W(t)+\frac{1}{\eps} (g^{\eps,\phi,x}_t)^-, \\
         g^\eps_0(\phi) &= \phi,\\
         g^{\eps,\phi,x}_t(0)&=g^{\eps,\phi,x}_t(\phi )(1)=0.
     \end{cases}
 \end{align}
 The coefficients satisfy the conditions of Theorem 4.1 in \cite{donati1993white} by choosing $\beta(\phi,x,t)(z)y=f(z,t,\omega)(y)$. Therefore, by Theorem \ref{Wohlgestellt} we have, for all $p\ge 1$,
 \begin{align*}
     &\lim_{\eps \searrow 0} \norm{g_t^{\eps,\phi,x} - g^{\phi,x}_t}_\infty = 0 \fastsicher, \\
     &\lim_{\eps \searrow 0} \E\Big(\sup_{0\le t\le T}\norm{g^{\eps,\phi,x}_t-g^{\phi,x}_t}_\infty^p\Big) = 0.
 \end{align*}
Moreover, the solutions $g^{\eps,\phi,x}_t$ are unique. 
  Moreover, from \eqref{AufleitungLip} and Assumption \hyperlink{(A3)}{(A3)} it follows that

\begin{align}\label{betaLip}
\begin{split}
    &\abs{\beta(\phi,x,t)(z)-\beta(\psi,y,t)(z)}\\
    =\; & \bigg\vert b^\prime\left(A[(g_t^{\phi,x},M^{\phi,x}_t)](z)),\lambda \circ A^{-1}([g_t^{\phi,x},M_t^{\phi,x}])(\cdot)\right)\\
    &\quad -b^\prime\left(A([g_t^{\psi,y},M_t^{\psi,y}](z)),\lambda \circ A^{-1}([g_t^{\psi,y},M_t^{\psi,y}])(\cdot)\right)\bigg\vert\\
    \le\; & C \left(\abs{A([g_t^{\phi,x},M^{\phi,x}_t])(z)-A([g_t^{\psi,y},M_t^{\psi,y}])(z)}+ \norm{g^{\phi,x}_t-g^{\psi,y}_t}_{\lp{2}}\right) \\
    \le\; & C\big(\norm{g^{\phi,x}_t-g^{\psi,y}_t}_{\lp{2}}+\abs{M^{\phi,x}_t-M^{\psi,y}_t}\big)
    \end{split}
\end{align}
for all $z\in [0,1]$.  Moreover, 
\begin{align*}
    \abs{\beta(\phi,x,t)(z)}\le C \fastsicher
 \end{align*}
for all $\phi\in \ccc_0([0,1])_{\ge 0}$, $x\in \R$, $t\in [0,\infty)$, and $z\in [0,1]$.
     \begin{bemerkung}
         It is important to note, that freezing the non-local part of the non-linearity is essential. In \cite{donati1993white} the authors show how to approximate solutions to the reflected heat equation by penalised SPDEs but they only show this for local nonlinearities. The proof cannot be copied to the case of non-local coefficents, since the proof in \cite{donati1993white} is crucially based on monotonicity, which cannot be realised with non-local coefficients.
     \end{bemerkung}
\begin{lemma}\label{Standardabsch}
    Let $\phi, \psi\in \ccc_0([0,1])_{\ge 0}$ and $x,y\in \R$. Then we have 
    \begin{enumerate}[label= \roman*)]
        \item For all $T>0$,
        \begin{align*}
            \E\Big(\sup_{0\le t \le T} \
            \norm{g^{\phi,x}_t}^p_{\lp{2}} \dd t \Big)\le C\big(1+\norm{\phi}^p_{\lp{2}}\big).
        \end{align*}
        \item For all $T>0$, and for $\tau^p_N= \inf\{t\ge 0:  \norm{g^{\phi,x}_t}^2_{\lp{2}}\ge N  \}$, 
        \begin{align*}
           &\E\Big(\sup_{0\le t \le T\land \tau_N} \norm{g^{\phi,x}_t-g^{\psi,y}_t}_{\lp{2}}^p+ \abs{M^{\phi,x}_t-M^{\psi,y}_t}^p \dd t \Big)\\
           \le\; & C\exp\Big(\frac{p}{2}NT\Big)\big(\abs{x-y}^p+\norm{\phi-\psi}^p_{\lp{2}}\big).       
        \end{align*}
    \end{enumerate}
\end{lemma}

\begin{proof}
   To prove i), fix $(\phi,x)\in \ccc_0([0,1])_{\ge 0}\times \R$ and $\beta(\phi,x,t)$. Consider the penalised equation \eqref{penalisierung}. We denote the solutions by $\Tilde{g}^{\eps,\psi}_t$ for $\psi\in \ccc_0([0,1])_{\ge 0}$.
   More precisely we consider 
   \begin{align*}
        \begin{cases}
         \dd \Tilde{g}^{\eps,\psi}_t &= \Delta g^{\eps,\psi,y}_t \dd t + \beta(\phi,x,t)g^{\eps,\psi}_t \dd t +\dd W(t)+\frac{1}{\eps} (g^{\eps,\psi,y}_t)^-, \\
         \Tilde{g}^{\eps,\psi}_0 &= \psi,\\
         \Tilde{g}^{\eps,\psi}_t(0)&=\Tilde{g}^{\eps,\phi,x}_t(\psi )(1)=0.
     \end{cases}
   \end{align*}
   
   Note that $\Tilde{g}^{\eps,\phi}_t=g^{\eps,\phi,x}_t$, but  $g^{\eps,\psi,y}_t\neq \Tilde{g}^{\eps,\psi}_t$ for  $\psi \neq \phi$, in general. 
   We can thus estimate 
   \begin{align*}
       &\norm{\Tilde{g}^{\eps,\phi}_t-\Tilde{g}^{\eps,\psi}_t}^2_{\lp{2}}\\
       =&\norm{\phi-\psi}^2_{\lp{2}} +2\int_0^t \skalarq{\beta(\phi,x,s)\Tilde{g}^{\eps,\phi}_s-\beta(\phi,x,s)\Tilde{g}^{\eps,\psi}_s,\Tilde{g}^{\eps,\phi}_s-\Tilde{g}^{\eps,\psi}_s}\dd s \\
       &- 2\int_0^t \norm{\nabla(\Tilde{g}^{\eps,\phi}_s-\Tilde{g}^{\eps,\psi}_s)}^2 \dd s\\
       &+ \frac{1}{\eps}\int_0^t \skalarq{(\Tilde{g}^{\eps,\phi}_s)^{-}-(\Tilde{g}^{\eps,\psi}_s)^{-},\Tilde{g}^{\eps,\phi}_s-\Tilde{g}^{\eps,\psi}_s} \dd s\\
       \le & \norm{\phi-\psi}^2_{\lp{2}} + C\int_0^t \norm{\Tilde{g}^{\eps,\phi}_s-\Tilde{g}^{\eps,\psi}_s}^2_{\lp{2}} \dd s.
   \end{align*}
   Note that $C$ depends neither on $\phi$ nor on $x$ since $b$ and thus $\beta$ is bounded. Hence, by Grönwall's inequality we get

\begin{align*}
    \sup_{0\le t\le T }\norm{\Tilde{g}^{\eps,\phi}_t-\Tilde{g}^{\eps,\psi}_t}^2_{\lp{2}}\le C \norm{\phi-\psi}^2_{\lp{2}},
\end{align*}
   and by letting $\eps \searrow 0$ we also obtain 
   \begin{align*}
       \sup_{0\le t\le T }\norm{g^{\phi,x}_t-\Tilde{g}^{\psi}_t}^2_{\lp{2}}\le C \norm{\phi-\psi}^2_{\lp{2}}
   \end{align*}
 for all $x\in \R$, where $\Tilde{g}^{\psi}_t$ is the solution to 
   \begin{align*}
       \begin{cases}
           d\Tilde{g}^{\psi}_t &= \Delta \Tilde{g}^{\psi}_t \dd t+ \beta(\phi,x,t)\Tilde{g}^{\psi}_t\dd t + \dd W(t) + \Tilde\eta, \\
           \Tilde{g}_0(\psi)&=\psi\ge 0,  \\
           \Tilde{g}^{\psi}_t(0)&=\Tilde{g}^{\psi}_t(1)=0, \;\Tilde{g}^{\psi}_t\ge 0,\\
               \langle \Tilde \eta, \Tilde g \rangle & =0.
       \end{cases}
   \end{align*}

Now we get for $p\ge 1$
\begin{align*}
    \E\big( \sup_{0\le t \le T} \norm{g^{\phi,x}_t}^p_{\lp{2}}\big)
    &\le C\,\E\big( \sup_{0\le t \le T} (\norm{g^{\phi,x}_t-\Tilde{g}_t^0}^p_{\lp{2}}
    +\norm{\Tilde{g}_t^0}^p_{\lp{2}}),\big)\\
    &= C\E\big(\sup_{0\le t\le T} (\norm{\Tilde{g}^{\phi}_t-\Tilde{g}^0_t}^p_{\lp{2}}+\norm{\Tilde{g}}^0_t)\big)\\
    &\le C\norm{\phi}^p_{\lp{2}}+ C\,\E\big( \sup_{0\le t \le T} \norm{\Tilde{g}_t^0}^p_{\lp{2}}\big)
\end{align*}
for all $(\phi,x)\in\ccc_0([0,1])  \times\R$.
The estimate now follows by showing that
\begin{align*}
    \E\big( \sup_{0\le t \le T} \norm{\Tilde{g}_t^0}^p_{\lp{2}}\big)\le C,
\end{align*}
where $C>0$ is independent of $(\phi,x)$. To this aim, consider 
\begin{align*}
    f_t(z)= \int_0^t\int_0^1 G_{t-s}(z,y)\beta(\phi,x,s)(y) \Tilde{g}_s^0(y)\dd y \dd s
    + \int_0^t\int_0^1 G_{t-s}(z,y)W(\dd y,\dd s).
\end{align*}
Then $z_t:=\Tilde{g}^0_t-f_t$ is the solution $(z,\eta)$ with obstacle $(f_t)_{t\ge0}$ in Lemma \ref{hindernisproblem}. Hence, by Lemma \ref{Hindernisabsch}, we can estimate
\begin{align*}
    \E\big(\sup_{0\le t \le T} \norm{\Tilde{g}^{0}_t}^p_\infty\big)
    &\le C\, \E\big(\sup_{0\le t \le T }\norm{f_t}^p_\infty\big) \\
    &\le C \int_0^T \E\big(\norm{\Tilde{g}^0_t}^p_\infty\big) \dd t
    +C\,  \E\big(\sup_{0\le t \le T}\norm{W_\Delta(t)}^p_\infty\big)\\
    &\le C\left(1+ \int_0^T \E\big(\norm{\Tilde{g}^0_t}^p_\infty\big) \dd t\right),
\end{align*}
where $W_\Delta(t)$ is the Ornstein–Uhlenbeck induced from the Dirichlet Laplacian $\Delta$ on $[0,1]$.  The finiteness of the moments follows from Lemma 5.21 in \cite{da2014stochastic} and the Kolmogorov continuity criterion (e.g.\ Theorem 1.8.1 in \cite{kunita2019stochastic}). 
In fact,  by Lemma 5.21 in \cite{da2014stochastic} there exist constants $C>0$, $\gamma\in (0,1)$ such that, for all $t,s\ge 0$ and $x,y\in [0,1]$,
\begin{align*}
    \E\big(\abs{W_{\Delta}(x,t)-W_{\Delta}(y,s)}^2\big)\le C (\abs{x-y}^2 +\abs{t-s}^2)^{\frac{\gamma}{2}},
\end{align*}
by Gaussianity we can thus conclude for all $t,s\ge 0$ and $x,y\in [0,1]$ that
\begin{align*}
    \E\big(\abs{W_{\Delta}(x,t)-W_{\Delta}(y,s)}^{2m}\big)\le \Tilde{C} (\abs{x-y}^2 +\abs{t-s}^2)^{\frac{m\gamma}{2}} 
\end{align*}
for all $m\in \N$ and some constant $\Tilde{C}>0$ possibly depending on $m\in  \N$. Hence we get from Theorem 1.8.1 in \cite{kunita2019stochastic} that
\begin{align*}
    \E\big(\sup_{0\le t \le T} \sup_{x\in [0,1]}\abs{W_\Delta(t,x)}^p\big)\le C 
\end{align*}
for some constant $C>0$ and all $p\ge 1$. 
Grönwall's inequality implies the desired result. 

We now prove ii). Take $\phi,\psi\in \ccc_0([0,1])_{\ge 0}$ and $x,y\in\R$.
We consider the penalised problem \eqref{penalisierung} with varying coefficients $\beta$ depending on the initial condition. Denote two solutions by $g^{\eps,\phi,x}_t$ and $g^{\eps,\psi,y}_t$. Hence 
\begin{align*}
      \begin{cases}
         \dd g^{\eps,\phi,x}_t &= \Delta g^{\eps,\phi,x}_t \dd t + \beta(\phi,x,t)g^{\eps,\phi,x}_t \dd t +\dd W(t)+\frac{1}{\eps} (g^{\eps,\phi,x}_t)^-, \\
         g^{\eps,\phi,x}_0 &= \phi,\\
         g^{\eps,\phi,x}_t(0)&=g^{\eps,\phi,x}_t(\phi )(1)=0.
     \end{cases}
\end{align*} 
and for $(\psi,y)$ respectively.

Then we get, similarly as in the proof of i), using \eqref{betaLip},
\begin{align*}
       \dd\norm{g^{\eps,\phi,x}_t-g^{\eps,\psi,y}_t}^2_{\lp{2}}
       &\le \skalarq{\beta(\phi,x,t)g^{\eps,\phi,x}_t-\beta(\psi,y,t)g^{\eps,\psi,y}_t,g^{\eps,\phi,x}_t-g^{\eps,\psi,y}_t}\dd t \\ 
       &=\skalarq{\beta(\psi,y,t)(g^{\eps,\phi,x}_t-g^{\eps,\psi,y}_t),g^{\eps,\phi,x}_t-g^{\eps,\psi,y}_t}\dd t\\
       &\quad+\skalarq{(\beta(\phi,x,t)-\beta(\psi,y,t))g^{\eps,\phi,x}_t,g^{\eps,\phi,x}_t-g^{\eps,\psi,y}_t} \dd t \\
       &\le C \norm{g^{\eps,\phi,x}_t-g^{\eps,\psi,y}_t}^2_{\lp{2}} \dd t\\
       &\quad+ C \big(\norm{g^{\phi,x}_t-g^{\psi,y}_t}_{\lp{2}}+ \abs{M^{\phi,x}_t-M^{\psi,y}_t} \big)\\
       &\qquad\times \big(\norm{g^{\eps,\phi,x}_t (\phi)}_{\lp{2}}\, \norm{g^{\eps,\phi,x}_t-g^{\eps,\psi,y}_t}_{\lp{2}} \big)  \dd t .
\end{align*}
Now, by letting $\eps \searrow 0$ and applying Young's inequality to the last term, we obtain 
\begin{align*}
\dd\norm{g^{\phi,x}_t-g^{\psi,y}_t}^2_{\lp{2}}
&\le C \big(\norm{g^{\phi,x}_t-g^{\psi,y}_t}^2_{\lp{2}}+ \abs{M^{\phi,x}_t-M^{\psi,y}_t}^2 \big) \dd t \\
&\quad+ C \norm{g^{\phi,x}_t-g^{\psi,y}_t}^2_{\lp{2}} \norm{g^{\phi,x}_t}^2_{\lp{2}} \dd t.
\end{align*}
Furthermore, we have 
\begin{align*}
    \dd \abs{M^{\phi,x}_t-M^{\psi,y}_t}^2
    \le  C \big(\norm{g^{\phi,x}_t-g^{\psi,y}_t}^2_{\lp{2}} +\abs{M^{\phi,x}_t-M^{\psi,y}_t}^2 \big) \dd t.
\end{align*}
Therefore Grönwall's inequality implies 
\begin{align*}
    &\norm{g^{\phi,x}_t-g^{\psi,y}_t}^2_{\lp{2}} + \abs{M^{\phi,x}_t-M^{\psi,y}_t}^2
    \\\le& \exp\Big(C\big(t+\int_0^t \norm{g^{\phi,x}_s}^2_{\lp{2}}\dd s\big)\Big)(\abs{x-y}^2+\norm{\phi-\psi}^2_{\lp{2}}) 
\end{align*}
and thus 
\begin{align*}
    \sup_{0\le t\le T \land \tau_N} &\norm{g^{\phi,x}_t-g^{\psi,y}_t}^p_{\lp{2}} 
    + \abs{M^{\phi,x}_t-M^{\psi,y}_t}^p
    \\\le& C\exp\Big(\frac{p}{2}NT\Big)\big(\abs{x-y}^p +\norm{\phi-\psi}^p_{\lp{2}}\big).
\end{align*}
\end{proof}

\begin{theorem}\label{Hauptresultat}
    For all $\theta\in (0,1)$ there exists a constant $C>0$ such that for all $T>0$ and all $(\phi,x),(\psi,y)\in \ccc_0([0,1])_{\ge 0}\times \R$,
   \begin{align*}
       &  \ent\big(P_T(\cdot,(\psi,y))\vert P_T(\cdot,(\phi,x))\big) \\
       \le\; &  C\Bigg(\frac{1}{T\land 1}\left(\rho^2+\rho^{\theta}+\rho\right)+\left(\rho^{1+\theta}+\rho^2\right)
    +(1+\norm{\phi}^2_{\lp{2}})^{1+\theta}\Big(\frac{T\land1}{1-\theta}\Big)^{\theta}\big(\log(1+\rho^{-1})\big)^{-\theta}\Bigg) \end{align*}
   where $\rho^2= \abs{x-y}^2+\norm{\phi-\psi}^2$.
\end{theorem}

\begin{proof} We can assume without loss of generality that $0<T\le1$ since $(P_t)_{t\ge 0}$ satisfies the semigroup property (see Corollary \ref{Markov}) thus the entropy is decreasing with respect to $t>0$. To see this fix $T>0$ and $t<T$, then we can deduce from duality and with Jensen's inequality, we get  
\begin{align*}
    &\ent\big(P_T(\cdot,(\phi,x)),P_T(\cdot,(\psi,y))\big)\\
     =& \sup_{0<F\in B_b(\ccc_0([0,1])\times \R)} \Big((P_T \log F) (\phi,x) - \log\big((P_TF)(\psi,y)\big)\Big)\\
    =&  \sup_{0<F\in B_b(\ccc_0([0,1])\times \R)} \Big((P_t P_{T-t} \log F)(\phi,x) - \log\big(P_tP_{T-t}F(\psi,y)\big)\Big)\\
    \le&   \sup_{0<F\in B_b(\ccc_0([0,1])\times \R)} \Bigg(P_t (\log\big(P_{T-t} F))(\phi,x)\big) - \log\big(P_t (P_{T-t}F)(\psi,y)\big)\Bigg)\\
    \le& \sup_{0<F\in B_b(\ccc_0([0,1])\times \R)}\Bigg( P_t \log\big(F(\phi,x)\big) - \log\big(P_tF(\psi,y)\big)\Bigg)\\
    =&\ent\big(P_t(\cdot,(\phi,x)),P_t(\cdot,(\psi,y))\big)
\end{align*}
for all $(\phi,x),(\psi,y)\in \ccc_0([0,1])_{\ge 0}\times \R$.

As announced we use a coupling argument in combination with the Girsanov transform.  Let $\phi,\psi\in \ccc_0([0,1])_{\ge 0}$ and $x,y\in \R$. 
Consider the equation
\begin{align*}
    \begin{cases}
        \dd \Tilde{g}^{\eps,\psi,y}_t&= (\Delta \Tilde{g}^{\eps,\psi,y}_t + \beta(\phi,x,t)g^{\eps,\phi,x}_t)\dd t + \dd W(t) + \frac{1}{\eps} (\Tilde{g}^{\eps,\psi,y}_t)^- \dd t \\
       &\quad - \frac{g^{\eps,\phi,x}_t-\Tilde{g}^{\eps,\psi,y}_t}{\xi(t)}  
      \ind{t<T}\, \dd t,  \\ 
      \Tilde{g}^{\eps,\psi,y}_t(0)&=\Tilde{g}^{\eps,\psi,y}_t(1)=0\\
        \Tilde{g}^{\eps,\psi,y}_0 &= \psi,\\[0.3em]
        \dd \Tilde{M}^{\psi,y}_t &= \int_0^1 \Tilde{\beta}(\phi,x,t)(M^{\phi,x}_t)(z) \dd z \dd t +\dd B_t \\
        &\quad - \dfrac{\Tilde{M}^{\psi,y}_t-M^{\phi,x}_t}{\xi(t)}  \ind{t<T}\, \dd t,\\
        \Tilde{M}^{\psi,y}_0&=y,
   \end{cases}
\end{align*}
where, in order to keep the notation simple, we write 
\[
\Tilde{\beta}(\phi,x,t)(\rho)(z)= b(A([g^{\phi,x}_t,\rho])(z),\lambda \circ (A[g^{\phi,x}_t,\rho](\cdot))^{-1}).
\]
Here $g^{\phi,x}_t$ and $M^{\phi,x}_t$ are the components of the solution system $(g^{\phi,x}_t,M^{\phi,x}_t)$. We will simply write $M_t$ instead of $M^{\phi,x}_t$ whenever this does not cause ambiguity. Moreover, let $\xi(t)=T-t$. It is clear that a solution $(\Tilde{g}^{\eps,\psi,y}_t,\Tilde{M}^{\psi,y}_t)_{t\in [0,T)}$ exists, and we will see that the solution can be extended beyond $t\ge T$ by $(g^{\eps,\phi,x}_t,M_t)_{t\ge T}$.
Now, by the chain rule,
\begin{align*}
    &\norm{\Tilde{g}^{\eps,\psi,y}_t-g^{\eps,\phi,x}_t}^2_{\lp{2}}  \\
   =& \norm{\phi-\psi}^2_{\lp{2}} + \int_0^{t} 2\skalarq{\Tilde{g}^{\eps,\psi,y}_s-g^{\eps,\phi,x}_s,\Delta(\Tilde{g}^{\eps,\psi,y}_s-g^{\eps,\phi,x}_s)} \dd s\\
   &\quad + \int_0^{t} \frac{2}{\eps}\skalarq{\Tilde{g}^{\eps,\psi,y}_s-g^{\eps,\phi,x}_s,(\Tilde{g}^{\eps,\psi,y}_s)^--(g_s^\eps)^-}\dd s
   -\int_0^{t} \frac{2}{\xi(s)} \norm{\Tilde{g}^{\eps,\psi,y}_s-g^{\eps,\phi,x}_s}^2_{\lp{2}} \dd s\\
   \le& \norm{\phi-\psi}^2_{\lp{2}}- \int_0^{t} \frac{2}{\xi(s)} \norm{\Tilde{g}^{\eps,\psi,y}_s-g^{\eps,\phi,x}_s}^2_{\lp{2}} \dd s,
\end{align*}
and 
\begin{align*}
    \abs{\Tilde{M}^{\psi,y}_t-M^{\phi,x}_t}^2= \abs{x-y}^2 
    - 2\int_0^{t} \frac{\abs{M_s^{\phi,x}-\Tilde{M}_s^{\psi,y}}^2}{\xi(s)} \dd s.
\end{align*}

Therefore,
\begin{align*}
      \norm{\Tilde{g}^{\eps,\psi,y}_t-g^{\eps,\phi,x}_t}_{\lp{2}}^2
      &\le  \norm{\phi-\psi}_{\lp{2}}^2\exp\Big(-\int_0^t \frac{2}{\xi(s)} \dd s\Big),\\
     \abs{\Tilde{M}^{\psi,y}_t-M^{\phi,x}_t}^2
     &\le \abs{x-y}^2\exp\Big(-\int_0^t \frac{2}{\xi(s)}\dd s\Big).
\end{align*}

Now observe that $\int_0^T \frac{1}{\xi(t)} \dd t = \infty$. Therefore, we can extend $(\Tilde{g}^{\eps}_t)_{t\in [0,T)}$ and $(\Tilde{M}^{\psi,y}_t)_{t\in [0,T)}$ up to time $T$ and thus beyond by 
\begin{align*}
    g^{\eps,\phi,x}_t= \Tilde{g}^{\eps,\psi,y}_t, \quad M^{\phi,x}_t= \Tilde{M}^{\psi,y}_t \fastsicher
\end{align*}
for all $t\ge T$.

Moreover, notice that 
\begin{align*}
    \dd\frac{\norm{\Tilde{g}_t^{\eps,\psi,y}-g^{\eps,\phi,x}_t}^2_{\lp{2}}}{\xi(t)}
    \le -\frac{\norm{\Tilde{g}_t^{\eps,\psi,y}-g^{\eps,\phi,x}_t}_{\lp{2}}^2}{\xi^2(t)}\underbrace{(2+\xi^\prime(t))}_{=1}\dd t
    \le 0.
\end{align*}

Integrating the inequality and rearranging the terms yields
\begin{align*}
   \int_0^{T-\delta} \frac{\norm{\Tilde{g}_t^{\eps,\psi,y}-g^{\eps,\phi,x}_t}_{\lp{2}}^2}{\xi^2(t)} \dd t 
   \le \frac{\norm{\phi-\psi}_{\lp{2}}^2}{\xi(0)} -\frac{\norm{\Tilde{g}_{T-\delta}^\eps-g^\eps_{T-\delta}}_{\lp{2}}^2}{\xi(T-\delta)}
   \le \frac{\norm{\phi-\psi}_{\lp{2}}^2}{\xi(0)}.
\end{align*} 

for all $\delta >0$ and thus letting $\delta \searrow 0$
\begin{align*}
     \int_0^{T} \frac{\norm{\Tilde{g}_t^{\eps,\psi,y}-g^{\eps,\phi,x}_t}_{\lp{2}}^2}{\xi^2(t)} \dd t\le \frac{\norm{\phi-\psi}_{\lp{2}}^2}{\xi(0)}.
\end{align*}
In exactly the same way we obtain
\begin{align*}
\int_0^T \frac{\abs{M^{\phi,x}_t-\Tilde{M}^{\psi,y}_t}^2}{\xi^2(t)} \dd t \le \frac{\abs{x-y}^2}{\xi(0)}.
\end{align*}
Furthermore, we have 
\begin{align*}
    \dd \Tilde{g}^{\eps,\psi,y}_t&= \Delta \Tilde{g}^{\eps,\psi,y}_t + \beta(\psi,y,t) \Tilde{g}^{\eps,\psi,y}_t \dd t + \dd \Tilde{W}(t) + \frac{1}{\eps}(\Tilde{g}^{\eps,\psi,y}_t)^- \dd t, \\
    \dd \Tilde{M}^{\psi,y}_t &= \int_0^1 \Tilde{\beta}(\psi,y,t)(\Tilde{M}^{\psi,y}_t)(z) \dd z \dd t  +\dd \Tilde{B}_t,
 \end{align*}
 with $\Tilde{B}$ and $\Tilde{W}$ defined by 
 \begin{align*}
     \Tilde{W}^\epsilon(t)&= W(t) + \int_0^t\big(\beta(\phi,x,s)g_s^{\eps,\phi,x} - \beta(\psi,y,s)\Tilde{g}^{\eps,\psi,y}_s\big)\dd s - \int_0^t \frac{g^{\eps,\phi,x}_s - \Tilde{g}^{\eps,\psi,y}_s}{\xi(s)}  \dd s,\\
     \Tilde{B}^\epsilon_t &= B(t) + \int_0^t\int_0^1 \big(\Tilde{\beta}(\phi,x,s)(M_s^{\phi,x})(z)-\Tilde{\beta}(\psi,y,s)(\Tilde{M}_s^{\psi,y})(z)\big) \dd z \dd s - \int_0^t \frac{\Tilde{M}_s^{\psi,y}-M_s^{\phi,x}}{\xi(s)}  \dd s.  
 \end{align*}

We will now show  that $\Tilde{W}^\epsilon$ cylindrical $L^2([0,1])$-white noise and $\tilde B$ is an independent Brownian motion on $(\Omega,\F{}{},\Q^\epsilon)$, where  $\Q^\epsilon$ is suitable Girsanov transform of $\W$. To this aim, we note  that 
\begin{gather*}
    \exp\left(\int_0^T \frac{\norm{g^{\eps,\phi,x}_s-\Tilde{g}^{\eps,\psi,y}_s}^2_{\lp{2}}}{\xi^ 2(s)} \dd s\right)\le C, \\
    \intertext{and by the boundedness of $\Tilde{\beta}$}
    \exp\left(\int_0^T \frac{\abs{M_s^{\phi,x}-\Tilde{M}_s^{\psi,y}}^2}{\xi^ 2(s)}+ \abs{\int_0^1 \Tilde{\beta}(\phi,x,s)(z)-\Tilde{\beta}(\psi,y,s)(z)\dd z}^2 \dd s\right) \le C. 
\end{gather*}
Furthermore, since $\norm{g^{\eps,\phi,x}_t-\Tilde{g}^{\eps,\psi,y}_t}_{\lp{2}}\le \norm{\phi-\psi}_{\lp{2}}$, we have
 \begin{align}\label{Girsanov}
     \begin{split}
     &\E\left(\exp\Big(\int_{t_{i-1}}^{t_i} \norm{\beta(\phi,x,s)g_s^{\eps}-\beta(\psi,y,s)\Tilde{g}^{\eps,\psi,y}_s}_{\lp{2}}^2 \dd s\Big)\right)\\
     \le\;& \E \Bigg( \exp\Big(2\int^{t_i}_{t_{i-1}} \norm{(\beta(\phi,x,s)-\beta(\psi,   y,s))g_s^{\eps}}_{\lp{2}}^2\dd s
     \\
     +\;& 2\int_{t_{i-1}}^{t_i} \norm{\beta(\psi,y,s)(g_s^{\eps}-\Tilde{g}^{\eps,\psi,y}_s)}_{\lp{2}}^2 \dd s\Big) \Bigg)\\
     \le\;& C\,\E\left(\exp\Big(C\int^{t_i}_{t_{i-1}} \norm{g^{\eps,\phi,x}_s}_{\lp{2}}^2 \dd s\Big)\right),
     \end{split}
 \end{align}
for any partition $(t_i)_{i=0,\dots,n}$ of $[0,T]$.

Define 
 \begin{align}\label{Dichte}
 \begin{split}
     \gamma_s^\epsilon&:= \frac{g^{\eps,\phi,x}_s - \Tilde{g}^{\eps,\psi,y}_s}{\xi(s)}-\beta(\phi,x,s)g_s^\eps - \beta(\psi,y,s)\Tilde{g}^{\eps,\psi,y}_s,\\
     \Tilde{\gamma}_s&:= \frac{M_s^{\phi,x}-\Tilde{M}_s^{\psi,y}}{\xi(s)}- \int_0^1 \big(\Tilde{\beta}(\phi,x,s)(M_s^{\phi,x})(z)-\Tilde{\beta}(\psi,y,s)(\Tilde{M}_s^{\psi,y})(z)\big) \dd z.
     \end{split}
 \end{align}
If there exists a partition $(t_{i})_{i=0,\dots,n}$ of $[0,T]$ such that the last term in \eqref{Girsanov} is finite, we can proceed as in Proposition 19 in the Appendix of \cite{da2013strong}: starting from
\begin{align*}
    &\E\Bigg(\exp\Big(\int_{t_{i-1}}^{t_i}\skalarq{\gamma^\epsilon_s,\dd W(s)}-\frac{1}{2}\int^{t_i}_{t_{i-1}}\norm{\gamma^\epsilon(s)}^2_{\lp{2}} \dd s + \int_{t_{i-1}}^{t_i}\Tilde{\gamma}_s\dd B_s -\frac{1}{2}\int_{t_{i-1}}^{t_i} \abs{\Tilde{\gamma}_s}^2 \dd s \Big)\Bigg) = 1,
\end{align*}
letting
\[
\mathcal{E}_s^t(\gamma^\epsilon)= \exp\Big(\int_s^t \skalarq{\gamma^\epsilon_r,\dd W(r)}-\frac{1}{2}\int_s^t\norm{\gamma^\epsilon_r}^2_{\lp{2}}\dd r+ \int_s^{t}\Tilde{\gamma}_r\dd B_r -\frac{1}{2}\int_{s}^{t} \abs{\Tilde{\gamma}_r}^2 \dd r \Big),
\]
we obtain
\begin{align*}
    \E( \mathcal{E}_0^T ) &= \E(\mathcal{E}^T_{t_{n-1}}\dots \mathcal{E}_0^{t_1}) \\
    &=\E\Big(\underbrace{\E\big(\mathcal{E}^T_{t_{n-1}}\vert \F{}{t_{n-1}}\big)}_{=1}\dots \mathcal{E}_0^{t_1}\Big)\\
    &=\dots = 1. 
\end{align*}
In this case we obtain a well defined probability measure $\Q^\eps_{\vert_{\F{}{T}}}$ on $(\Omega, {\F{}{T}})$ via the Radon-Nikodym density 
\begin{align*}
\dd \Q^\epsilon_{\vert_{\F{}{T}}}= \exp\Big(\int_{0}^T\skalarq{\gamma^\epsilon_s,\dd W(s)}-\frac{1}{2}\int^{T}_{0}\norm{\gamma^\epsilon(s)}^2_{\lp{2}} \dd s + \int_{0}^{T}\Tilde{\gamma}_s\dd B_s -\frac{1}{2}\int_{0}^{T} \abs{\Tilde{\gamma}_s}^2 \dd s\Big) \dd \W .
\end{align*}

To find such a partition note that 
\begin{align*}
    g^{\eps,\phi,x}_t(z)&= \int_0^1 G_t(z,y) g_0(y)\dd y + \int_0^t \int_0^1 G_{t-s}(z,y)\beta(\phi,x,s)(y)g^{\eps,\phi,x}_s(y) \dd y \dd s\\
    &\quad+ \int_0^t \int_0^1 G_{t-s}(z,y)  W(\dd y,\dd s ) 
    +\frac{1}{\eps} \int_0^t \int_0^1 G_{t-s}(z,y)(g^{\eps,\phi,x}_s(y))^- \dd y \dd s.  
\end{align*}
Therefore, Grönwall's lemma yields
\begin{align*}
    \sup_{0\le t\le T} \norm{g^{\eps,\phi,x}_t}^2_{\lp{2}} \le C(\eps,T) + \sup_{0\le t\le T}\norm{W_\Delta(t)}^2_{\lp{2}},
\end{align*}
where $W_\Delta$ is the stochastic convolution related to $\Delta$. By Fernique's theorem, or more precisely by Proposition 18 in \cite{da2013strong}, there exists $\delta>0$ such that 
\begin{align*}
    \E\big(\exp\big(\delta\sup_{0\le t \le T }\norm{W_\Delta(t)}_{\lp{2}}^2\big)\big) <\infty.
\end{align*}

Hence there exists a partition $(t^n_{i})_{i=0,\dots,n}$, $n\in \N$, satisfying the needed conditions, and $\Q^\eps_{\vert_{\F{}{T}}}$ is in fact well defined. Moreover, from Girsanov's theorem we conclude that  $\Tilde{W}$ is $L^2([0,1])$-cylindrical white noise and $\Tilde{B}$ is an independent Brownian motion on $(\Omega,\F{}{T},\Q^\eps_{\F{}{T}})$. \\

 Let us estimate the entropy between $\Q^\eps_{\vert_{\F{}{T}}}$ and $\W_{\vert_{\F{}{T}}}$. Using the stopping time  $\tau_N(\phi):= \inf\{t\ge 0: \norm{g^{\phi,x}_t}^2_{\lp{2}} \ge N\}$ it holds that 
\begin{align*}
     &\ent(\Q_{\F{}{T}}^\eps\vert \W_{\F{}{T}}) \notag\\ \notag
    \le&C\E\left(\int^T_0 \norm{\beta(\phi,x,t)g_t^{\eps,\phi,x} - \beta(\psi,y,t)\Tilde{g}^{\eps,\psi,y}_t}_{\lp{2}}^2 \dd t \right)\\\notag
    &\quad+ \E\left(\int_0^T \frac{\norm{g_t^{\eps,\phi,x}-\Tilde{g}^{\eps,\psi,y}_t}_{\lp{2}}^2}{\xi(t)^2} \dd t + \int_0^T\abs{\int_0^1 \Tilde{\beta}(\phi,x,t)(M^{\phi,x}_t)(z)-\Tilde{\beta}(\psi,y,t)(\Tilde{M}^{\psi,y}_t)(z) \dd z}^2 \dd t\right) \\\notag
    &\quad+   \E\left(\int_0^T \frac{\abs{M^{\phi,x}_t-\Tilde{M}^{\psi,y}_t}^2}{\xi(t)^2} \dd t\right) \\\notag
    \le&   \frac{C}{T}\Big(\norm{\psi-\phi}^2_{\lp{2}}+ \abs{x-y}^2\Big) \\
    &\quad+ C\Bigg[
    \E\left(\int^T_0 \norm{\beta(\phi,x,t)g_t^{\eps,\phi,x} - \beta(\psi,y,t)\Tilde{g}^{\eps,\psi,y}_t}_{\lp{2}}^2 \dd t \right)\\\notag
    &\qquad+ \E\left( \int_0^T\abs{\int_0^1 \Tilde{\beta}(\phi,x,t)(M^{\phi,x}_t)(z)-\Tilde{\beta}(\psi,y,t)(\Tilde{M}^{\psi,y}_t)(z) \dd z}^2 \dd t\right)\Bigg]\\\notag
    \le&  \frac{C}{T}\Big(\abs{x-y}^2+\norm{\phi-\psi}^2_{\lp{2}}\Big)+C\Bigg(\E\left(\int_0^T \norm{\beta(\psi,y,t)(g^{\eps,\phi,x}_t-\Tilde{g}_t^{\eps,\psi,y})}_{\lp{2}}^2\dd t\right)\\\notag
    &\qquad+ \E\left(\int_0^T \norm{g^{\eps,\phi,x}_t(\beta(\phi,x,t)-\beta(\psi,y,t))}_{\lp{2}}^2 \dd t\right)\\\notag
    &\qquad+ \E\left(\int_0^T \int_0^1 \abs{\Tilde{\beta}(\phi,x,t)(M^{\phi,x}_t)(z)-\Tilde{\beta}(\psi,y,t)(\Tilde{M}^{\psi,y}_t)(z)}^2\dd z \dd t  \right) \Bigg).\\ \notag 
    \le&  \frac{C}{T}\Big(\abs{x-y}^2+\norm{\phi-\psi}^2_{\lp{2}}\Big)+C\Bigg( +\E\left(\int_0^T \norm{\beta(\psi,y,t)(g^{\eps,\phi,x}_t-\Tilde{g}_t^{\eps,\psi,y})}_{\lp{2}}^2\dd t\right)\\\notag
    &\qquad+ \E\left(\int_0^T \norm{g^{\eps,\phi,x}_t(\beta(\phi,x,t)-\beta(\psi,y,t))}_{\lp{2}}^2 \dd t(\ind{\tau_n(\phi)<T}+\ind{\tau_n(\phi)\ge T})\right)\\\notag
    &\qquad+ \E\left(\int_0^T \int_0^1 \abs{\Tilde{\beta}(\phi,x,t)(M^{\phi,x}_t)(z)-\Tilde{\beta}(\psi,y,t)(\Tilde{M}^{\psi,y}_t)(z)(\ind{\tau_n(\phi)<T}+\ind{\tau_n(\phi)\ge T})}^2\dd z \dd t  \right) \Bigg).\notag\\
    &\le  \frac{C}{T}\Big(\abs{x-y}^2+\norm{\phi-\psi}^2_{\lp{2}}\Big)+C\Bigg( \E\left(\int_0^T \norm{\beta(\psi,y,t)(g^{\eps,\phi,x}_t-\Tilde{g}_t^{\eps,\psi,y})}_{\lp{2}}^2\dd t\right)\notag \\ \notag
    &\qquad+ \E\left(\int_0^T \norm{g^{\eps,\phi,x}_t(\beta(\phi,x,t)-\beta(\psi,y,t))}_{\lp{2}}^2 \dd t\ind{\tau_n(\phi)<T}\right)\\\notag
     &\qquad+ \E\left(\int_0^{T\land \tau_N(\phi)} \norm{g^{\eps,\phi,x}_t(\beta(\phi,x,t)-\beta(\psi,y,t))}_{\lp{2}}^2 \dd t\right)\\\notag
    &\qquad+ \E\left(\int_0^{T\land \tau_N(\phi)} \int_0^1 \abs{\Tilde{\beta}(\phi,x,t)(M^{\phi,x}_t)(z)-\Tilde{\beta}(\psi,y,t)(\Tilde{M}^{\psi,y}_t)(z)}^2\dd z \dd t  \right) \notag \\
    &\qquad+ \E\left(\int_0^{T} \int_0^1 \abs{\Tilde{\beta}(\phi,x,t)(M^{\phi,x}_t)(z)-\Tilde{\beta}(\psi,y,t)(\Tilde{M}^{\psi,y}_t)(z)}^2\dd z \dd t\ind{\tau_N(\phi)<T}  \right)
    \Bigg).\notag\\
    &\le \frac{C}{T}\Big(\abs{x-y}^2+\norm{\phi-\psi}^2_{\lp{2}}\Big)+ C\Bigg( \E\left(\int_0^T \norm{\beta(\psi,y,t)(g^{\eps,\phi,x}_t-\Tilde{g}_t^{\eps,\psi,y})}_{\lp{2}}^2\dd t\right)\notag \\ \notag
    &\qquad+ \E\left(\int_0^T \norm{g^{\eps,\phi,x}_t}_{\lp{2}}^{2p} \dd t\right)^{\frac{1}{p}}\W(\tau_n(\phi)<T)^{\frac{p-1}{p}}\\\notag
     &\qquad+ \E\left(\int_0^{T\land \tau_N(\phi)} \norm{g^{\eps,\phi,x}_t(\beta(\phi,x,t)-\beta(\psi,y,t))}_{\lp{2}}^2 \dd t\right)\\\notag
    &\qquad+ \E\left(\int_0^{T\land \tau_N(\phi)} \int_0^1 \abs{\Tilde{\beta}(\phi,x,t)(M^{\phi,x}_t)(z)-\Tilde{\beta}(\psi,y,t)(\Tilde{M}^{\psi,y}_t)(z)}^2\dd z \dd t  \right) \notag \\
    &\qquad+ \W(\tau_N(\phi)<T)
    \Bigg).\notag \\ 
     &\le \frac{C}{T}\Big(\abs{x-y}^2+\norm{\phi-\psi}^2_{\lp{2}}\Big)+ C\Bigg( \abs{x-y}^2+\norm{\phi-\psi}^2_{\lp{2}}\\
     &\qquad+\E\left(\int_0^T \norm{\beta(\psi,y,t)(g^{\eps,\phi,x}_t-\Tilde{g}_t^{\eps,\psi,y})}_{\lp{2}}^2\dd t\right)\notag \\ \notag
    &\qquad+ \E\left(\int_0^T \norm{g^{\eps,\phi,x}_t}_{\lp{2}}^{2p} \dd t\right)^{\frac{1}{p}}\W(\tau_n(\phi)<T)^{\frac{p-1}{p}}\\\notag
     &\qquad+ \E\left(\int_0^{T\land \tau_N(\phi)} \norm{g^{\eps,\phi,x}_t(\beta(\phi,x,t)-\beta(\psi,y,t))}_{\lp{2}}^2 \dd t\right)\\\notag
    &\qquad+ \E\left(\int_0^{T\land \tau_N(\phi)} \int_0^1 \abs{\Tilde{\beta}(\phi,x,t)(M^{\phi,x}_t)(z)-\Tilde{\beta}(\psi,y,t)(\Tilde{M}^{\psi,y}_t)(z)}^2\dd z \dd t  \right) \notag \\
    &\qquad+ \W(\tau_N(\phi)<T)
    \Bigg).\notag
\end{align*}

for $p>1$, where the last inequality follows from the Hölder inequality and boundedness of $\beta$ and $\Tilde{\beta}$. Letting $\eps \searrow 0$, applying Lemma \ref{Standardabsch} and using the Lipschitz properties of $\beta$ and $\Tilde{\beta}$, we obtain 
\begin{align*}
 &\limsup_{\eps\searrow 0} \ent(\Q^\eps_{\vert_{\F{}{T}}}\vert \W_{\vert_{\F{}{T}}})\\
 \le& \frac{C}{T}\Bigg(\abs{x-y}^2+\norm{\phi-\psi}^2_{\lp{2}}\Bigg)\\
 &\qquad+  C\Bigg(\abs{x-y}^2+\norm{\phi-\psi}^2_{\lp{2}}\\
 &\qquad+\E\Big(\int_0^{T\land \tau_N(\phi)} \big(\norm{g^{\phi,x}_t-g^{\psi,y}_t}_{\lp{2}}^2+\abs{M^{\phi,x}_t-M^{\psi,y}_t}^2\big) \norm{g^{\phi,x}_t}_{\lp{2}}^2 \dd t  \Big)\\   
    &\qquad+ \E\Big(\int_0^{T\land \tau_N(\phi)} \norm{g^{\phi,x}_t-g^{\psi,y}_t}^2_{\lp{2}}\dd t  \Big)\\
    &\qquad+ \E\big(\sup_{0\le t\le T} \norm{g^{\phi,x}_t}_{\lp{2}}^{2p}\big)^{\frac{1}{p}}\W(\tau_N(\phi) <T)^{\frac{p-1}{p}}+\W(\tau_N(\phi)<T)\Bigg)\\
      \le& \frac{C}{T}(\abs{x-y}^2+\norm{\phi-\psi}^2) \\
      &\qquad+ C\Bigg( \abs{x-y}^2+\norm{\phi-\psi}^2_{\lp{2}}\\
      &\qquad+ \E\Big(\sup_{0\le t\le T\land \tau_N(\phi)}\big(\norm{g^{\phi,x}_t-g^{\psi,y}_t}^2_{\lp{2}}+ \abs{M^{\phi,x}_t-M^{\psi,y}_t}^2\big) \norm{g^{\phi,x}_t}_{\lp{2}}^2\Big)\\
      &\qquad+ \E\Big(\sup_{0\le t \le T\land \tau_N(\phi)} \norm{g^{\phi,x}_t-g^{\psi,y}_t}^2_{\lp{2}}\Big)
      + \W(\tau_N(\phi)<T)\\
      &\qquad+ \E\big(\sup_{0\le t\le T} \norm{g^{\phi,x}_t}_{\lp{2}}^{2p}\big)^{\frac{1}{p}}\W(\tau_N(\phi) <T)^{\frac{p-1}{p}}\Bigg)\\
      \le&  C \Bigg(\Big(\frac{1}{T}+1+\exp(NT)+N\exp(NT)\Big)(\abs{x-y}^2+\norm{\phi-\psi}^2_{\lp{2}}) \\
      &\qquad+ \W(\tau_N(\phi)<T)
      +\E\big(\sup_{0\le t\le T} \norm{g^{\phi,x}_t}_{\lp{2}}^{2p}\big)^{\frac{1}{p}}\W(\tau_N(\phi) <T)^{\frac{p-1}{p}}\Bigg).
\\    \le& C\Bigg(\Big(\frac{1}{T}+1+(1+N)\exp(NT)\Big)(\abs{x-y}^2+\norm{\phi-\psi}^2_{\lp{2}}) \\
    &\quad+ \W(\tau_N(\phi)<T)^{\frac{p-1}{p}}
    +\E\big(\sup_{0\le t\le T} \norm{g^{\phi,x}_t}_{\lp{2}}^{2p}\big)^{\frac{1}{p}}\W(\tau_N(\phi) <T)^{\frac{p-1}{p}}\Bigg),
\end{align*}
Now for $\theta\in (0,1)$ choose $p=\frac{1}{1-\theta}$ such that  $\frac{p-1}{p}= \theta$, and set $N=\frac{1-\theta}{T}\log(1+\rho^{-1})$, where $\rho^2= \abs{x-y}^2+\norm{\phi-\psi}^2$. Observe that 
\begin{align*}
    \W(\tau_N(\phi)<T)\le \frac{C(1+\norm{\phi}^2_{\lp{2}})}{N}  
\end{align*}
and use Lemma \ref{Standardabsch}. Now we get, using $\log(1+x)\le x$ and $(1+x)^q\le 1+x^q$ for $x>0,q<1$, 
\begin{align*}
     & \limsup_{\eps \searrow 0}\ent(\Q_{\vert_{\F{}{T}}}^\eps\vert\W_{\vert_{\F{}{T}}})\\
    \le& C\Big(\frac{\rho^2}{T}+\big(1+\tfrac{1-\theta}{T}\log(1+\rho^{-1})\big)(1+\rho^{-1})^{1-\theta}\rho^2 \\
    &\qquad + (1+\norm{\phi}^2_{\lp{2}})^{\frac{p-1}{p}}\Big(\frac{T}{1-\theta}\Big)^{\theta}\big(\log(1+\rho^{-1})\big)^{-\theta}\\
    &\qquad + (1+\norm{\phi}^2_{\lp{2}})(1+\norm{\phi}^2_{\lp{2}})^{\theta}\Big(\frac{T}{1-\theta}\Big)^{\theta}\big(\log(1+\rho^{-1})\big)^{-\theta}\Big)\\
    \le& C\Big(\frac{\rho^2}{T}+\big(\rho^2+\tfrac{1-\theta}{T}\rho\big)(1+\rho^{-1})^{1-\theta   }\\
    &\qquad +(1+\norm{\phi}^{1+\theta}_{\lp{2}})^{2}\Big(\frac{T}{1-\theta}\Big)^{\theta}\big(\log(1+\rho^{-1})\big)^{-\theta}\Big)\\
        \le& C\Big(\frac{\rho^2}{T}+\big(\rho^2+\tfrac{1-\theta}{T}\rho\big)(1+\rho^{\theta-1})\\
    &\qquad +(1+\norm{\phi}^2_{\lp{2}})^{2}\Big(\frac{T}{1-\theta}\Big)^{\theta}\big(\log(1+\rho^{-1})\big)^{-\theta}\Big)\\
   \le& C\Bigg(\frac{1}{T}\left(\rho^2+\rho^{\theta}+\rho\right)+\left(\rho^{1+\theta}+\rho^2\right)\\
    &\qquad +(1+\norm{\phi}^{1+\theta}_{\lp{2}})^{2}\Big(\frac{T}{1-\theta}\Big)^{\theta}\big(\log(1+\rho^{-1})\big)^{-\theta}\Bigg),
\end{align*}

Moreover, consider the map $\Phi^\eps_T:\Omega \to \lp{2}([0,1])_{\ge 0}\times \R$ defined by 
\begin{align*}
    \Phi^\eps_T(\omega)= (g^{\eps,\phi,x}_T,M^{\phi,x}_T)(\omega),
\end{align*}
then we obtain by entropy duality
\begin{align*}
    \ent\big( P^\eps_T(\cdot,(\psi,y))&\vert P^\eps_T(\cdot,(\phi,x))\big)
    = \sup_{f\in \lp{\infty}(\lp{2}([0,1])\times \R)}\E\Big(f\big((g^{\eps,\phi,x}_t,M_T^{\phi,x})\big)\Big)\\
    &\quad - \log\left(\E\big(\exp(\psi(g^{\eps,\phi,x}_T,M_T^{\phi,x}))\big)\right)\\
    &=\sup_{f\in \lp{\infty}(\lp{2}([0,1])\times \R)} \E\big(f((g^{\eps,\phi,x}_T,M_T^{\phi,x}))\big)\\
    &\quad - \log\left(\E_{\Q^\eps}\big(\exp(f(g^{\eps,\phi,x}_T,M_T^{\phi,x}))\big)\right)\\
    &=\sup_{f\in \lp{\infty}(\lp{2}([0,1])\times \R)} \E\big(f(\Phi^\eps_T(\omega))\big)
    - \log\bigg(\E_{\Q^\eps}\big(\exp(f(\Phi_T^\eps(\omega)))\big)  \bigg)\\
    &\le \sup_{f\in \lp{\infty}(\Omega_{\vert_{\F{}{T}}})} \E(f) - \log(\E_{\Q^\eps}(\exp(f)))\\
    &= \ent(\Q_{\vert_{\F{}{T}}}^\eps\vert\W_{\vert_{\F{}{T}}}). 
\end{align*}

We can now deduce from lower the joint lower  semicontinuity of the entropy with respect to weak convergence that for $T \in [0,1]$
\begin{align*}
    \ent\big(&P_T(\cdot,(\psi,y))\vert P_T(\cdot,(\phi,x)) \big)
    \le \liminf_{\eps \to 0} \ent(\Q_{\vert_{\F{}{T}}}^\eps\vert \W_{\vert_{\F{}{T}}} )\\
   &\le C(T)\Bigg(\frac{1}{T}\left(\rho^2+\rho^{\theta}+\rho\right)+\left(\rho^{1+\theta}+\rho^2\right)
    +(1+\norm{\phi}^{1+\theta}_{\lp{2}})^{2}\Big(\frac{T}{1-\theta}\Big)^{\theta}\big(\log(1+\rho^{-1})\big)^{-\theta}\Bigg),
\end{align*}
which finishes the proof.
 \end{proof}

From the entropy estimate in Theorem \ref{Hauptresultat}, we can immediately deduce the strong Feller property. In fact, we actually get a slightly better result.

\begin{corollary}\label{StarkFeller}
    Let $F:\lp{2}([0,1])\times \R \to \R$ be bounded and measurable, then 
    \begin{align*}
        P_TF: &\ccc_0([0,1])_{\ge 0}\subset \lp{2}([0,1])_{\ge 0} \times \R \to \R\\
        & (\phi,x) \mapsto \E(F(g_t^{\phi,x},M_t^{\phi,x}))
    \end{align*} 
    is locally uniformly and continuous and continuosly extendable to a map 
    \begin{align*}
        \Tilde{P}_TF: \lp{2}([0,1])_{\ge 0 } \times \R \to \R  
    \end{align*}
\end{corollary}
\begin{proof}Take $M\in \N$, $\phi,\psi \in \ccc_0([0,1])_{\ge 0}$ such that 
$\norm{\phi}_{\lp{2}},\norm{\psi}_{\lp{2}} \le M$ and $x,y\in \R$.
 Note first that Theorem \ref{Hauptresultat} and Pinsker's inequality imply
\begin{align*}
   & \abs{\E(F(g^{\phi,x}_T,M^{\phi,x}_T))-\E(F(g^{\psi,y}_T,M^{\psi,y}_T))}\\
   \le& \norm{F}_{\infty}d_{TV}(P_T(\cdot,(\phi,x)),P_T(\cdot,(\psi,y))) \\
   \le& C \left(\ent(P_T(\cdot,(\psi,y))\vert P_T(\cdot,(\phi,x)) )\right)^{\frac{1}{2}}\\
   \le&   C\Bigg(\frac{1}{T\land 1}\left(\rho^2+\rho^{\theta}+\rho\right)+\left(\rho^{1+\theta}+\rho^2\right)
    +(1+\norm{\phi}^2_{\lp{2}})^{1+\theta}\Big(\frac{T\land 1}{1-\theta}\Big)^{\theta}\big(\log(1+\rho^{-1})\big)^{-\theta}\Bigg)^{\frac{1}{2}}
\end{align*}
    for any $\theta\in (0,1)$. Hence 
   \begin{align*}
       \abs{\E(F(g^{\phi,x}_T,M^{\phi,x}_T))-\E(F(g^{\psi,y}_T,M^{\psi,y}_T))}\le G_T(\rho)
   \end{align*} 
     such that $G_T:[0,\infty)\to [0,\infty)$, $G(0)=0$ and $G$ continuous. Therefore $P_TF$ is uniformly continuous on
     \[
\big(\ccc_0([0,1])\cap \{\phi \in\lp{2}([0,1])_{\ge 0}: \norm{\phi}_{\lp{2}}\le M \}\big)\times \R
\]
for all $M\in \N$.
Hence the function can be uniquely extended to a uniformly continuous function on 
\begin{align*}
     \{\phi \in\lp{2}([0,1])_{\ge 0}: \norm{\phi}_{\lp{2}}\le M \}\big)\times \R
\end{align*}
for all $M\in \N$. Hence the extension exists on $\lp{2}([0,1])_{\ge 0}\times \R$ and the result is shown.

\end{proof}

Strong uniqueness of autonomous stochastic evolution equations usually implies the Markov property. However, in this case we have an additional measure term $\eta$ which, at first glance, transforms the equation into a non-autonomous one. Since the measure term completely depends on the solutions $(g^{\phi,x}_t,M^{\phi,x}_t)$, we can proceed to prove the Markov property. Denote by $(g^{s,\phi,x}, M^{s,\phi,x}, \eta^s)$ the solutions to the equation \eqref{Loesung}, but started at time $s\ge 0$ in the usual way. Note that one can obtain the same uniqueness result as in Theorem \ref{Wohlgestellt}.
\begin{corollary}\label{Markov}
    For all $t,s\ge 0$, $\phi \in \ccc_0([0,1])_{\ge 0}$ and all $x\in \R$ we have 
    \begin{align*}
        (g^{s,\phi,x}_t,M^{s,\phi,x}_t) \overset{d}{=} (g^{\phi,x}_{t-s},M^{\phi,x}_{t-s}).
    \end{align*}
\end{corollary}
\begin{beweis}
   Let $\psi \in \ccc_c((0,1))$. Then, by Definition \ref{Loesung}, one can deduce that
    \begin{align*}
        \skalarq{g^{s,\phi,x}_{(t+s)-s},\psi}_{\lp{2}} &= \skalarq{\phi,\psi}_{\lp{2}} + \int^t_s \skalarq{g^{s,\phi,x}_r,\Delta\psi}_{\lp{2}} \dd r \\
        &\quad+ \int_s^t \skalarq{b^\prime(A([g^{s,\phi,x}_r,M^{s,\phi,x}_r]),\mu^s_r) g^{s,\phi,x}_r,\psi}_{\lp{2}} \dd r + \int_s^t \int_0^1 \psi(x) \eta(\dd x,\dd r)\\
        &\quad+ \int_s^t \int_0^1 \psi(x) W(\dd x ,\dd r),\\
        M^{s,\phi,x}_{(t+s)-s}&= x + \int_s^t \int_0^1 b(A([g^{\phi,x}_r,M^{\phi,x}_r])(z),\mu_r) \dd z \dd r + B_t-B_s\\
         \skalarq{g^{s,\phi,x}_{(t+s)-s},\psi}_{\lp{2}} &= \skalarq{\phi,\psi}_{\lp{2}} + \int_0^{t-s} \skalarq{g^{s,\phi,x}_{r+s},\Delta\psi}_{\lp{2}} \dd r \\
        &\quad+ \int_0^{t-s} \skalarq{b^\prime(A([g^{s,\phi,x}_{r+s},M^{s,\phi,x}_{r+s}]),\mu^s_{r+s}) g^{s,\phi,x}_{r+s},\psi}_{\lp{2}} \dd r \\
        &\quad+ \int_0^{t-s} \int_0^1 \psi(x) \Tilde{\eta}(\dd x,\dd r)
        + \int_0^{t-s } \int_0^1 \psi(x) \Tilde{W}(\dd x ,\dd r),\\
         M^{s,\phi,x}_{(t+s)-s}&= x + \int_0^{t-s} \int^1_0 b(A([g^{\phi,x}_{r+s},M^{\phi,x}_{r+s}])(z),\mu_r) \dd z \dd r + \Tilde{B}_{t-s},
    \end{align*}
   where $\Tilde{W}(\cdot,t)= W(\cdot,t+s)-W(\cdot,s)$, $\Tilde{B}_t= B_{t+s}-B_s$ and $\Tilde{\eta}(A,B)= \eta(A,B+s)$ for $A\in \mathcal{B}([0,1])$, $B\in \mathcal{B}(\R_+)$. Hence the triple $(g^{s,\phi,x}_{s+t},M^{s,\phi,x}_{s+t},\Tilde{\eta})_{t\ge 0}$ must, by uniqueness in law, coincide with $(g^{\phi,x}_t,M^{\phi,x}_t,\eta)_{t\ge 0}$ in law, and thus the result above is obtained. 
\end{beweis}

Note that the flow property (for initial conditions in $(\phi,x)\in \ccc_0([0,1])_{\ge 0}\times \R$) can also be easily shown by uniqueness of the solutions. Therefore, together with Theorem \ref{Hauptresultat}, one can obtain the Markov property in exactly the same way as in \cite{prevot2007concise}, Proposition 4.3.5. Thus Corollary \ref{StarkFeller} yields the strong Feller property.   

\begin{theorem}\label{StarkeFellerErw}
    The Markov process $((g^{\phi,x},M^{\phi,x}))_{\phi \in \ccc_0([0,1])_{\ge 0},x\in \R}$ is extendable to a Markov process $((\Tilde{g}^{\phi,x},\Tilde{M}^{\phi,x}))_{\phi\in \lp{2}([0,1])_{\ge 0},x\in \R}$ such that $(\Tilde{g}^{\phi,x},\Tilde{M}^{\phi,x})$ is a strong Feller process with state space $\lp{2}([0,1])_{\ge 0}\times \R $. Moreover, there a constant $C \in (0,\infty)$ such that for all $T>0$ and all $(\phi,x),(\psi,y)\in \lp{2}([0,1])_{\ge 0}\times \R$
    \begin{align*}
       &\ent\big(P_T(\cdot,(\psi,y))\vert P_T(\cdot,(\phi,x))\big) \\
       \le\; & C \Bigg(\frac{1}{T \land 1 }\left(\rho^2+\rho^{\theta}+\rho\right)+\left(\rho^{1+\theta}+\rho^2\right)
    +(1+(\norm{\phi}^2_{\lp{2}})^{1+\theta}\Big(\frac{T\land 1 }{1-\theta}\Big)^{\theta}\big(\log(1+\rho^{-1})\big)^{-\theta}\Bigg)
    \end{align*}
 holds, where $\rho^2=\abs{x-y}^2+\norm{\phi-\psi}^2$.
\end{theorem}

 \begin{proof}
     Let $\phi\in \lp{2}([0,1])_{\ge 0}$ and consider a sequence $(\phi_n)_{n\in \N}$ such that $\norm{\phi_n-\phi}_{\lp{2}}\to 0$. Then we can conclude that $(g^{\phi_n,x}_t,M^{\phi_n,x}_t)$ is Cauchy  in $\lp{2}(\Omega,\ccc([0,T],\lp{2}([0,1])\times \R))$ by the following argument: choose $M\in \N$ such that for all $n\ge M$ we have $\norm{\phi_n}_{\lp{2}([0,1])}\le C$. Let $\eps > 0$. By Lemma \ref{Standardabsch} one can choose $N\in \N$ such that for all $m\ge n\ge M$,
 \begin{align*}
     \E\Big(\sup_{0\le t\le T}\norm{g_t^{\phi_m,x}-g_t^{\phi_n,x}}^4_{\lp{2}} + \abs{M_t^{\phi_m,x}-M_t^{\phi_n,x}}^4 \Big)^{\frac{1}{2}}\W(\tau_N(\phi_n)\le T)^{\frac{1}{2}} \le \eps,
 \end{align*}
 where $\tau_N(\phi_n)= \inf\{t\ge 0: \norm{g_t^{\phi_n,x}}_{\lp{2}} \ge N\}$. Hence 
 \begin{align*}
     &\E\Big(\sup_{0\le t\le T} \norm{g_t^{\phi_m,x}-g_t^{\phi_n,x}}_{\lp{2}}^2 +\abs{M_t^{\phi_m,x}-M_t^{\phi_n,x}}^2\Big)\\
     =& \E\Big(\sup_{0\le t\le T} \big(\norm{g_t^{\phi_m,x}-g_t^{\phi_n,x}}_{\lp{2}}^2 +\abs{M_t^{\phi_m,x}-M_t^{\phi_n,x}}^2\big)\big(\ind{\{\tau_N(\phi_n) \le T\}} +\ind{\{\tau_N(\phi_n)>T\}}\big)\Big)\\ 
     \le & \E\Big(\sup_{0\le t\le T\land \tau_N(\phi_n)} \norm{g_t^{\phi_m,x}-g_t^{\phi_n,x}}_{\lp{2}}^2 +\abs{M_t^{\phi_m,x}-M_t^{\phi_n,x}}^2\Big)\\
     &\quad +\E\Big(\sup_{0\le t\le T}\norm{g_t^{\phi_m,x}-g_t^{\phi_n,x}}_{\lp{2}}^4 + \abs{M_t^{\phi_m,x}-M_t^{\phi_n,x}}^4 \Big)^{\frac{1}{2}}\W(\tau_N(\phi_n)\le T)^{\frac{1}{2}}.
 \end{align*}
By Lemma \ref{Standardabsch} we get 
 \begin{align*}
  \lim_{m\ge n\to \infty }   \E\Big(\sup_{0\le t\le T} \norm{g_t^{\phi_m,x}-g_t^{\phi_n,x}}_{\lp{2}}^2 +\abs{M_t^{\phi_m,x}-M_t^{\phi_n,x}}^2\Big)\le \eps.
 \end{align*}
 Since $\eps>0$ was arbitrary, we can thus define $(\Tilde{g}^{\phi,x},\Tilde{M}^{\phi,x})$ with $(\phi,x)\in\lp{2}([0,1])_{\ge 0}\times \R$ as the limit of $\left((g^{\phi_n,x},M^{\phi_n,(x)})\right)_{n\in \N}$, where $\phi_n\to \phi$ and $\phi_n\in \ccc_0([0,1])_{\ge 0}$ for all $n\in \N$. One can thus show that $P_tf\in \ccc_b(\lp{2}([0,1])_{\ge 0}\times \R)$ whenever $f\in \ccc_b(\lp{2}([0,1])\times \R)$ and is Lipschitz on $\lp{2}([0,1])\times \R$.   

 Furthermore, let $0\le s_1\le \dots\le s_n \le t$ and $\psi,f_1,\dots, f_n \in \ccc_b(\lp{2}([0,1])\times \R)$ be Lipschitz. By a monotone class argument it suffices to show 
 \begin{align*}
     &\E\Big(\psi\big((\Tilde{g}^{\phi,x}_t,\Tilde{M}^{\phi,x}_t)\big) f_1\big((\Tilde{g}^{\phi_x}_{s_1},\Tilde{M}^{\phi,x}_{s_1})\big)\dots f_n\big((\Tilde{g}^{\phi,x}_{s_n},\Tilde{M}^{\phi,x}_{s_n})\big)\Big)\\
     =&\E\Big(P_{t-s_n}\psi\big((\Tilde{g}^{\phi,x}_{s_n},\Tilde{M}^{\phi,x}_{s_n})\big) f_1\big((\Tilde{g}^{\phi,x}_{s_1},\Tilde{M}^{\phi,x}_{s_1})\big)\dots f_n\big((\Tilde{g}^{\phi,x}_{s_n},\Tilde{M}^{\phi,x}_{s_n})\big)\Big),
 \end{align*}
 which follows for $(\phi,x)\in \ccc_0([0,1])_{\ge 0}\times \R$ by classical arguments, i.e.\ uniqueness and Corollary \ref{Markov}. By the approximation argument from above, the equality extends to all $(\phi,x)\in \lp{2}([0,1])_{\ge 0}\times \R$. The last statement follows from Theorem \ref{Hauptresultat} and the fact that, for all $\phi\in \lp{2}([0,1])_{\ge 0}$,
 \begin{align*}
     P_T(\cdot,(\phi^n,x)) \to P_T(\cdot,(\phi,x))
 \end{align*}
 in the weak topology whenever $(\phi^n)_{n\in \N}\in \ccc_0([0,1])_{\ge 0}$ with $\norm{\phi^n-\phi}_{\lp{2}}\to 0$. 
  \end{proof}

Finally, we can conclude the main result.

\begin{theorem}  
Under the assumptions \hyperlink{(A2)}{(A2)} and \hyperlink{(A3)}{(A3)}, the system \eqref{System3} is well posed for initial conditions $M_0 \in \R$, $0\leq g_0 \in C_0([0,1])$. The family of solutions induces a unique Markov process on $ \PP_2^1(\R)$ such that, for all bounded measurable $F:  \PP_2^1(\R) \mapsto \R$, the map   
\[
  \PP_2^1(\R)\ni \mu \mapsto  \mathbb E \big(F(\mu_t)\,\big|\mu=\mu\big) \in \R
\]
is continuous, locally uniformly with respect to the metric $\rho$. 
\end{theorem}

\begin{beweis}
The Markov process $(\mu_t)_{t\ge 0}$ is defined by $\mu_t= \lambda \circ (A[(g_t,M_t)] (\cdot))^{-1}$ for all $t> 0$ and $\mu = \lambda \circ\left(A\big([\tfrac{\partial}{\partial u}F^{\mu}],\skalarq{F^{\mu}}\big)(\cdot)\right)^{-1}$, where $F^{\mu}$ is the inverse CDF of $\mu\in \PP_2^1(\R)$. The result now follows from Theorem \ref{Wohlgestellt}, Theorem \ref{Hauptresultat} and Theorem \ref{StarkeFellerErw}. 
\end{beweis}

\renewbibmacro{in:}{} 

\nocite{lasry2006jeux,lasry22006jeux,lasry2007mean}  \nocite{huang2003individual,huang2007nash}.

\nocite{delarue2020master}  
\printbibliography
\end{document}